\documentclass[letterpaper, 10 pt, conference]{ieeeconf}  

\usepackage{color}

\IEEEoverridecommandlockouts                              

\usepackage{graphics} 
\usepackage{epsfig} 
\usepackage{mathptmx} 
\usepackage{times} 
\usepackage{amsmath} 
\usepackage{amssymb}  
\usepackage{epsfig,epstopdf}
\allowdisplaybreaks

\title{\LARGE \bf
Mixed-Triggered Reliable Control for Singular Networked Cascade Control Systems with Randomly Occurring Cyber Attack
}
\author{Sathishkumar Murugesan and Yen-Chen~Liu
\thanks{This work was supported in part by the Ministry of Science and Technology, Taiwan, under grants MOST 107-2811-E-006-537 and MOST 108-2636-E-006-007.}
\thanks{S. Murugesan and Y.-C. Liu are with the Department of Mechanical Engineering, National Cheng Kung University, Tainan 70101, Taiwan,
e-mail:\texttt{\{sathishmaths21@gmail.com, yliu@mail.ncku.edu.tw\}.}}}

\newtheorem{theorem}{Theorem}[section]
\newtheorem{lemma}[theorem]{Lemma}
\newtheorem{definition}[theorem]{Definition}

\begin{document}

\maketitle
\thispagestyle{empty}
\pagestyle{empty}

\begin{abstract}
In this paper, the issue of mixed-triggered reliable dissipative control is investigated for singular networked cascade control systems (NCCSs) with actuator saturation and randomly occurring cyber attacks. { In order to utilize the limited communication resources effectively, a more general mixed-triggered scheme is established which includes both schemes namely time-triggered and event-triggered in a single framework. In particular, two main factors are incorporated to the proposed singular NCCS model namely, actuator saturation and randomly occurring cyber attack, which is an important role to damage the overall network security.
By employing Lyapunov-Krasovskii stability theory, a new set of sufficient conditions in terms of linear matrix inequalities (LMIs) is derived to guarantee the singular NCCSs to be admissible and strictly $(\mathcal{Q}, \mathcal{S}, \mathcal{R})$-dissipative.
Subsequently, a power plant boiler-turbine system based on a numerical example is provided to demonstrate the effectiveness of the proposed control scheme.}
\end{abstract}
\section{INTRODUCTION}\label{sec:intro}
In more recent years, cascade control systems have grabbed remarkable attention from researchers due to its numerous applications in many technical areas, e.g. power plants, neural networks, chemical reactors, and networked control systems.
To mention specifically, this control algorithm utilizes a pair of control loops, where the second loop (secondary loop) is embedded with the first loop (primary loop).
The first loop is responsible to the system stability, and the second loop is able to quickly eliminate the disturbances~\cite{NCCS, NCCS3}.
Nowadays much significance is given to real-time networked control systems due to the unique features of reduced weight and power requirements, low cost, simple installation and maintenance, higher flexibility, and easy reconfigurability.
Therefore, it is vital to analysis the combination of the cascade control systems and networked control systems, known as networked cascade control systems.

Singular systems, also known as differential algebraic or descriptor systems, have been vastly discussed by researchers in the control society.
The emerging area attracts significant attentions because they arise naturally in various scientific fields, such as chemical process, economic systems, power systems, electrical networks, and mechanical systems~\cite{ss1,ss2}.
Furthermore, despite of several advantages of singular NCCSs, it inevitably leads to a few major issues, such as time-varying delay, unpredictable actuator fault and exogenous disturbances which formulates the analysis and design of singular NCCSs very complicated.

It is essential to improve the communication constraints so as to transmit the large control information by the limited bandwidth communication network.
To ensure the performance of NCSs in practice and to overcome the drawback of inadequate bandwidth resources, event-triggering communication schemes have been predominantly considered.
In contrast to the conventional communication of time-triggering, the former can be opted, as it facilitates to release the sample data packets into the network more efficiently.
Recently, various works on event-triggering communication have been discussed~\cite{tig1,tig2,event}.
In the practical point of view, it is necessary and important to consider the combination of both time- and event-triggering, which is known as a mixed triggered scheme~\cite{base}.

It is executed by making use of random switch connecting time- and event-triggered. Accordingly the implementation of mixed triggered scheme results in enhancement of the system performance and that minimize the network transmission at a time.
Moreover, the reliable control design has gained significance since it has the capability of maintaining system stability and holds performance of the required systems even though the actuator faults are present~\cite{reliable1,reliable2,relconf}.
Furthermore, the most of the real-time actuators can distribute input signals of bounded amplitude only because of physical restrictions which may lead to the actuator saturation phenomenon~\cite{base, sat1}.
Recently, more focus has been paid to cyber attack in networked control systems due to its vigorously opening-up property of data-transmission channels through an unsecured communication network medium which are vulnerable to be disrupted by ambiguity.
Additionally, it can be classified as deception attacks~\cite{datt}, replay attacks~\cite{ratt} and denial of service attacks~\cite{DoS1}.
For example, the controller design problem is addressed for networked systems with stochastic cyber attack based on a mixed-triggering scheme~\cite{cattack}.

%

Up to now, only a few works have been done previously on control synthesis and stability analysis of singular NCCSs.
Some vital glitches in controller schemes such as stochastic cyber attack, actuator faults, actuator saturation and disturbances have not so far been considered for singular NCCSs.
The aforementioned works motivate and draw our research interest towards establishing a mixed-triggered control design with respect to unexpected actuator faults and saturation for singular NCCSs under external disturbances and randomly occurring cyber attack.
Some notable features of this paper are encapsulated with unique aspects as given below:
\begin{itemize}
\item[(i)] A novel primary and secondary mixed-triggered reliable dissipative controller design is developed for the singular NCCSs with actuator saturation, unexpected actuator faults and randomly occurring cyber attack.
\item[(ii)] In the proposed strictly $(\mathcal{Q}, \mathcal{S}, \mathcal{R})$ dissipativity results consolidates the results of $H_\infty$, passivity and mixed $H_\infty$ and passivity, which makes the considered issue more general one.
\item[(iii)] Based on the Lyapunov stability theorem, a novel mixed-triggered reliable dissipative control is developed, which ensures the admissibility of the proposed system.
\end{itemize}

\section{PRELIMINARIES AND PROBLEM FORMULATION}\label{p1-sec2}

In this paper, singular NCCSs with actuator saturations and stochastic cyber attacks are considered with mixed-triggered reliable dissipative control.
As shown in Fig.~\ref{p1-fig1}, the structure of the cascade control system contains two loops. The inner loop is made up of secondary plant $\Sigma_2$, secondary sensor $S_2$, secondary controller.
The outer loop is composed of the primary plant $\Sigma_1$, primary sensor $S_1$, primary controller, and the actuator $A$.
A network is assumed to connect the sensor and primary controller in the cascade control system, which constitutes the singular NCCS studied in this paper.

\begin{figure}[!htb]
	\begin{minipage}[b]{.70\textwidth}
		\includegraphics[width=8.8cm]{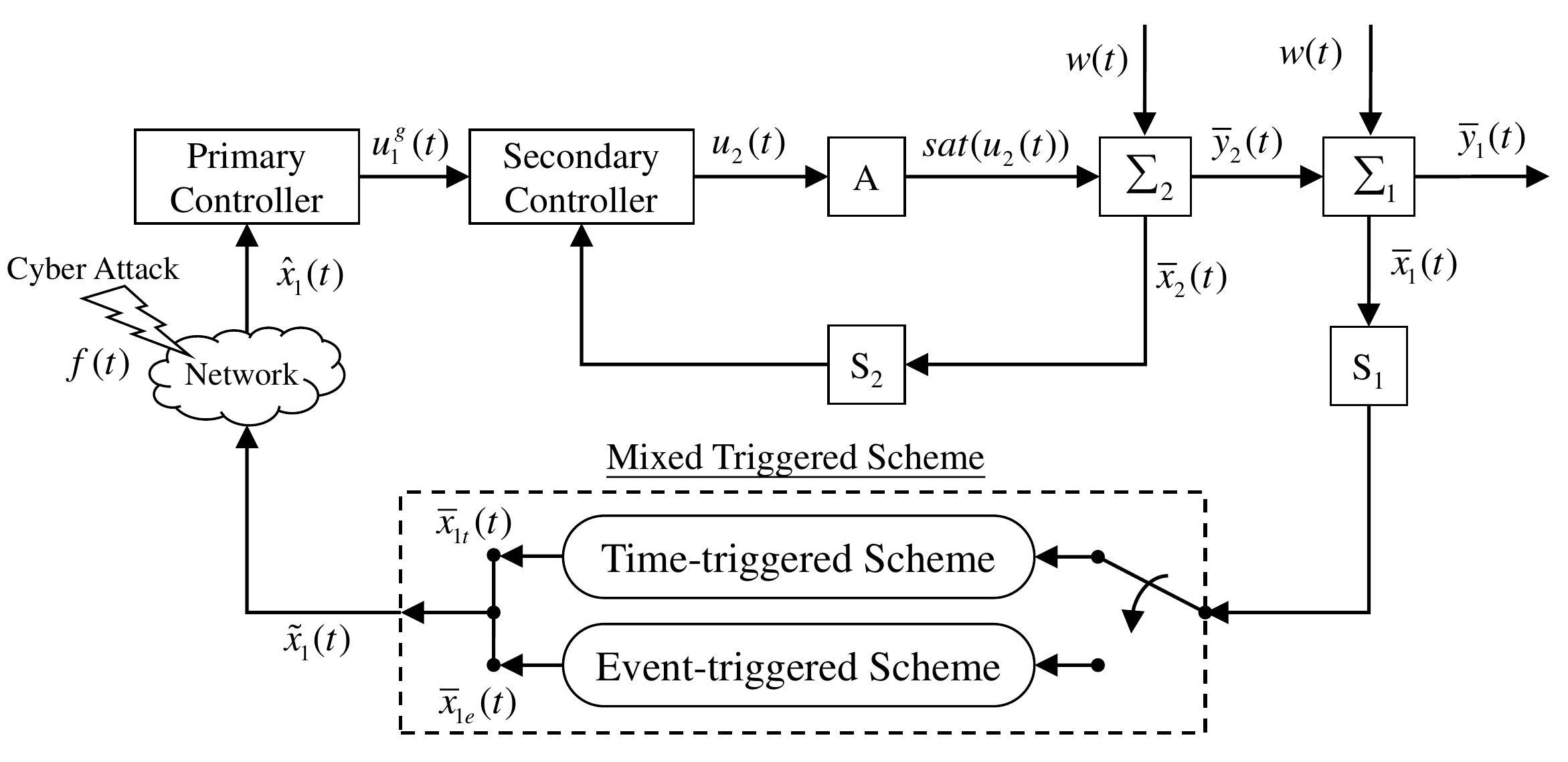}
	\end{minipage}
	\caption{Configuration diagram of the proposed reliable mixed-triggered control for singular NCCSs}\label{p1-fig1}
\end{figure}

\subsection{Cascade Control System}

We consider that the primary plant $\Sigma_1$ of the NCCSs is described by
\begin{eqnarray}
{\mbox{$\Sigma_1$:}} \left\lbrace
\begin{array}{l}\label{p1-eq1}
\dot{\bar{x}}_{1}=A_{1}\bar{x}_{1}+B_{1}\bar{y}_{2},\\
\bar{y}_{1}=C_{1}\bar{x}_{1}+D_1w,
\end{array}
\right.
\end{eqnarray}
where $x_{1}$ is the state vector of $\Sigma_1$, $y_1$ is the output vector, $A_1$, $B_1$, $C_1$ and $D_1$ are known constant matrices with appropriate dimensions, $y_2$ is the output of the $\Sigma_2$.
In the proposed system, the secondary plant $\Sigma_2$ is modeled by using a class of singular systems with state delay that
\begin{eqnarray}
&&\hspace{-.3in}
{\mbox{$\Sigma_2$:}}\left\lbrace
\begin{array}{l}\label{p1-eq2}
\mathcal{E}\dot{\bar{x}}_{2}=A_{2}\bar{x}_{2}+A_{3}\bar{x}_{2}(t-\theta(t))
+B_{2}sat(u_{2})+B_{3}w,\\
\bar{y}_{2}=C_{2}\bar{x}_{2}+D_2w,\\
\bar{x}_2=\phi,\quad t\in[-\bar{\theta}, 0],
\end{array}
\right.
\end{eqnarray}
where $\bar{x}_2$ and $u_2$ are the state vector and control input vector of the $\Sigma_2$, $w$ is the exogenous disturbances which belongs to $\mathcal{L}_{2}[0,\infty)$, $\bar{y}_2$ is the output of the secondary plant, and $A_2$, $A_3$, $B_2$, $B_3$, $C_2$ and $D_2$ are known constant matrices with appropriate dimensions.
Additionally, the matrix $\mathcal{E}$ may be singular and it is assumed that $rank(\mathcal{E})=r\leq n$, and the function $\phi$ is the initial condition
defined on $[\bar{\theta}, 0]$.
The term $\theta(t)$ denotes the time-varying delay and satisfies $0\leq\theta(t)\leq \bar{\theta}$ with $\dot{\theta}(t)\leq\lambda<1,$  where $\lambda$ is a positive integer representing maximum time delay.
Moreover, $sat(u_{2})$ denotes the saturation function of the actuator in $\Sigma_2$, which is defined as $sat(u_2)=
\begin{bmatrix}
sat(u_2^1) & sat(u_2^2) & \cdots & sat(u_2^m)
\end{bmatrix}^T\in\mathbb{R}^m$
 where
\begin{eqnarray}
sat(u_2^i)=
\left\lbrace
\begin{array}{cl}\label{p1-eq3}
\Xi_i, & u_2^i>\Xi_i \\
u_2^i, &-\Xi_i\leq u_2^i\leq \Xi_i,\ i=1, 2, \cdots, m \\
-\Xi_i, & u_2^i< -\Xi_i
\end{array}
\right.
\end{eqnarray}
with $\Xi_i$ denotes the known upper limits of actuator saturation constraints in $\Sigma_2$.
The output of the saturation function $sat(u_2)$ can be split into two sections which includes linear and nonlinear.
Therefore, the following saturation model will be considered~\cite{base} in this paper
\begin{eqnarray}\label{p1-eq4}
sat(u_2)=u_2-\Psi(u_2).
\end{eqnarray}
Furthermore, the dead-zone nonlinearity function $\Psi(u_2)$ satisfies the following condition that there exists $\epsilon\in(0,1)$ with $\epsilon=\max\{\epsilon_1, \epsilon_2, \cdots, \epsilon_m\}$ such that
\begin{eqnarray}\label{p1-eq5}
\epsilon u_2^Tu_2\geq \Psi^T(u_2)\Psi(u_2).
\end{eqnarray}
In order to construct the suitable controller for the primary and secondary plant, we design the following state feedback controller as
\begin{eqnarray}
\left\lbrace
\begin{array}{lll}\label{p1-eq6}
u_1^g=\mathbb{G}u_1=\mathbb{G}\mathcal{K}_1\hat{x}_1,\\
u_2=u_1^g+\mathcal{K}_2\bar{x}_2,
\end{array}
\right.
\end{eqnarray}
where $\mathcal{K}_1$ and $\mathcal{K}_2$ are the state feedback control gain matrices of primary and secondary controllers, respectively.
In~\eqref{p1-eq6}, $\hat{x}_1$ is the actual state input of primary controller.
The reliable control law $u_{1}^{g}$ is defined as $u_{1}^{g}= \mathbb{G}u_{1}$,
where $\mathbb{G}$ is the actuator fault matrix and it is possible to defined in the following matrix form that $\mathbb{G}=\mbox{diag}\{g_{1},g_{2},\dots,g_{m}\}$ with $0\leq\underline{g}_{k}\leq g_{k}\leq\overline{g}_{k}\leq 1,$ $k=1, 2, \dots, m$.
When $g_k=0$, the $k^{th}$ actuator fully fails, whereas $g_k=1$ means that the $k^{th}$ actuator functions normally.
For the sake of simplicity, we define $\bar{\mathbb{G}}=\mbox{diag}\{\bar{g}_{1},\bar{g}_{2},\dots,\bar{g}_{m}\},$
$\underline{\mathbb{G}}=\mbox{diag}\{\underline{g}_{1},\underline{g}_{2},\dots,\underline{g}_{m}\}$, $\mathbb{G}_0=\frac{\bar{\mathbb{G}}+\underline{\mathbb{G}}}{2}$, $\mathbb{G}_1=\frac{\bar{\mathbb{G}}-\underline{\mathbb{G}}}{2}$.
Thus, the fault matrix $\mathbb{G}$ can be expressed in the following form
\begin{align}\label{p1-eq7}
\mathbb{G}=\mathbb{G}_0+\mathbb{G}_1\bar{\Sigma},
\end{align}
{where $\bar{\Sigma}=\mbox{diag}\{l_1,l_2,\dots, l_m\}\in \mathbb{R}^{m\times m}$ with $-1 \leq l_k\leq 1$. }
\subsection{Problem formulation}
It should be noted that a more general mixed-triggered control
scheme is established to reduce the burden of network bandwidth, which is modeled
in a probabilistic way by utilizing random variable to express the switch between time and event-triggered schemes.
When the signal is passed through the time-triggered scheme, the sampled measurements are periodic and it will be transmitted in time.
Furthermore, the sequence of transmitting instant is denoted by $t_kh \ (k=1, 2, \cdots)$, where $h$ is a sampling period, $t_k \ (k=1, 2, \cdots)$ is a sequence set of positive integers, namely, $\{t_1, t_2, \cdots\}=\{1, 2, \cdots\}$. Roughly speaking, when the latest transmitting instant is $t_kh$ and the following transmitting instant is
$t_{k+1}h=t_kh+h$.
Therefore, $\eta_{t_k}$ denotes the network-induced
time-delay of sensor measurement sampled at the instant $t_kh$.

The singular NCCSs in Fig. \ref{p1-fig1} modeled with $\zeta(t)=t-t_kh$ and the sensor measurement of $\Sigma_1$ as
\begin{eqnarray}\label{p1-eq8}
\bar{x}_{1t}=\bar{x}_1(t_kh)=\bar{x}_1(t-\zeta(t)), \ t\in [t_kh+\zeta_{t_k}, t_{k+1}h+\zeta_{t_{k+1}}),
\end{eqnarray}
where $\zeta(t)\in[0, \zeta_2]$, $\zeta_2$ is the upper bound of network-induced delay.
Suppose the random switch signal passes through the channel of event-triggered scheme, the periodically sampled measurements will be transmitted to the communication network only possible when they violate the condition of triggering.
Thus, the sequence of transmitting instant can be described in the event-triggering condition that
\begin{eqnarray}\label{p1-eq9}
e_k^TWe_k\leq \mu \bar{x}_1^T(t_kh+lh)W\bar{x}_1(t_kh+lh),
\end{eqnarray}
where $e_k(t)=\bar{x}_1(t_kh)-\bar{x}_1(t_kh+lh)$, and $W>0$ is a matrix, $l=1, 2, \cdots$ and scalar $\mu\in[0,1)$.
The latest transmitted signal is represented by $x_1(t_kh)$ at the latest triggering time $t_kh$. Whether the latest sampled signals $x(t_kh+lh)$ are possible to delivered or not depends on the condition~\eqref{p1-eq9}.
For improvement, the time intervals can be converted into many subintervals, which can be written in the following form that $[t_kh+\zeta_{t_k}, t_{k+1}h+\zeta_{t_k+1})=\bigcup_{l=0}^{d}[t_kh+lh+\zeta_{t_{k}+l},t_kh+lh+h+\zeta_{t_{k}+1+l}]$, where $l=1, \cdots, d$, $d=t_{k+1}-t_k-1$.
By letting $d(t)=t-t_kh-lh$, it is easy to obtain the limit of $d(t)$ as $0<\zeta_{{t_k}+1}\leq d(t)\leq d_2$, in which $d_2 \triangleq  h+\zeta_{{t_k}+l+1}$. Therefore, the sensor measurement can be modeled as
\begin{eqnarray}\label{p1-eq10}
&&\hspace{-.25in} \bar{x}_{1e}=\bar{x}_1(t_kh)=\bar{x}_1(t-d(t))+e_k, t\in [t_kh+\zeta_{t_k}, t_{k+1}h+\zeta_{t_k+1}). \nonumber
\end{eqnarray}

In order to frame the mixed-triggered scheme, we introduce the stochastic variable $\alpha(k)$ that satisfies the Bernoulli distributed white sequences with
$\mathbb{E}\{\alpha(t)\}=\bar{\alpha}, \
\mathbb{E}\{\alpha(t)-\bar{\alpha}\}=0, \
\mbox{and} \
\mathbb{E}\{(\alpha(t)-\bar{\alpha})^2\}=\sigma^2$.
Based on the above discussion, the modified mixed-triggered control scheme is expressed by
\begin{eqnarray}\label{p1-eq11}
\tilde{x}_1=\alpha(t)\bar{x}_{1t}+(1-\alpha(t))\bar{x}_{1e}.
\end{eqnarray}
It is noted that the cyber attacks are established by a random manner.
The nonlinear function $f(\bar{x}_1)$ is introduced to express the phenomena of cyber attack and its time delay is assumed to satisfy $0<\tau(k)\leq \tau_2$.
Moreover, the variable $\beta(k)$ is mutually independent and Bernoulii distributed white sequence, which is utilized to govern the randomly occurring cyber attacks.
\begin{eqnarray}
&&\beta(t)=\left\lbrace
\begin{array}{lll}
1, \quad {\mbox{cyber attacks are executed}}\\
0, \quad {\mbox{transmission is normal}}
\end{array}.
\right.
\end{eqnarray}
Subsequently, with the mixed-triggered scheme and randomly occurring cyber attack, the primary control input can be represented as
\begin{eqnarray}\label{p1-eq12}
\hat{x}_1=\beta(t)f(\bar{x}_1(t-\tau(t)))+(1-\beta(t))\tilde{x}_1,
\end{eqnarray}
where
$
\Pr\{\beta(t)=1\}=\mathbb{E}\{\beta(t)\}=\bar{\beta}, \
\Pr\{\beta(t)=0\}=1-\mathbb{E}\{\beta(t)\}=1-\bar{\beta},\
\mathbb{E}\{\beta(t)-\bar{\beta}\}=0, \
\mbox{and} \
\mathbb{E}\{(\beta(t)-\bar{\beta})^2\}=\delta^2.
$
Hence, the primary controller of~\eqref{p1-eq6} can be rewritten as
\begin{eqnarray}\label{p1-eq13}
&&\hspace{-.28in} u_1^g=\mathbb{G}u_1=(1-\beta(t))(1-\alpha(t))\Big[\mathbb{G}\mathcal{K}_1\bar{x}_1(t-d(t))+\mathbb{G}\mathcal{K}_1e_k\Big]\nonumber\\
&&\hspace{-.28in}+\beta(t)\mathbb{G}\mathcal{K}_1f(\bar{x}_1(t-\tau(t)))+(1-\beta(t))\alpha(t)\mathbb{G}\mathcal{K}_1\bar{x}_1(t-\zeta(t)). \nonumber
\end{eqnarray}
Therefore, from~\eqref{p1-eq1},~\eqref{p1-eq2},~\eqref{p1-eq4} and $u_1^g$, the closed-loop control system can be described by
\begin{eqnarray}
\left\lbrace
\begin{array}{llll}\label{p1-eq14}
\hspace{0.25cm}\dot{\bar{x}}_{1}=A_{1}\bar{x}_{1}+B_{1}C_{2}\bar{x}_{2}+B_{1}D_2w,\\
\mathcal{E}\dot{\bar{x}}_{2}=A_{2}\bar{x}_{2}+A_{3}\bar{x}_{2}(t-\theta(t))\\
\hspace{0.95cm}+(1-\beta(t))\alpha(t)B_2\mathbb{G}\mathcal{K}_1\bar{x}_1(t-\zeta(t))\\
\hspace{0.95cm}+(1-\beta(t))(1-\alpha(t))B_2\mathbb{G}\mathcal{K}_1\bar{x}_1(t-d(t))\\
\hspace{0.95cm}+(1-\beta(t))(1-\alpha(t))B_2\mathbb{G}\mathcal{K}_1e_k\\
\hspace{0.95cm}+\beta(t)B_2\mathbb{G}\mathcal{K}_1f(\bar{x}_1(t-\tau(t)))\\
\hspace{0.95cm}+B_2\mathcal{K}_2\bar{x}_2-B_2\Psi(u_2)+B_3w,\\
\hspace{0.26cm}\bar{y}_{1}=C_{1}\bar{x}_{1}+D_1w.
\end{array}
\right.
\end{eqnarray}
Here, we recall definitions and lemmas are more essential to get the required results.
\begin{definition}\label{p1-def1}\cite{NCCS}
The pair $(\mathcal{E},A_2)$ is said to be regular if $\det(s\mathcal{E}-A_2)$ is not identically zero and
 the pair $(\mathcal{E},A_2)$ is said to be impulse free if $\deg(\det(s\mathcal{E}-A_2))=$rank$(\mathcal{E}).$  Further,
the unforced singular system is said to be regular and impulse free, if the pair $(\mathcal{E},A_2)$ is regular and impulse free.
\end{definition}
\begin{definition}\cite{prof1}
Given scalar $\gamma>0$ matrices $\mathcal{Q}$, $\mathcal{R}$ and $\mathcal{S}$ with $\mathcal{Q}$ and $\mathcal{R}$ real symmetric, the closed-loop systems \eqref{p1-eq14} is strictly $(\mathcal{Q}, \mathcal{S}, \mathcal{R})$ dissipative if for $t>0$ under zero initial state, the following condition is satisfied:
\begin{align*}
<\bar{y}_1,\mathcal{Q}\bar{y}_1>_t&+2<\bar{y}_1, \mathcal{S} w>_t+<w, \mathcal{R}w>_t
\geq \gamma<w,w>_t.
\end{align*}
Without loss of generality, we assume that the matrix $\mathcal{Q}\leq 0$, $\bar{\mathcal{Q}}=\sqrt{-\mathcal{Q}}$ and $<u,v>_t=\int_{0}^{t}u^Tvdt$.
\end{definition}
\begin{lemma}\label{p1-schur}\cite{NCCS}
Given constant matrices $\Omega_1,\ \Omega_2$ and  $\Omega_3$ with appropriate dimensions, where
$\Omega_1=\Omega_1^T<0$ and $\Omega_2=\Omega_2^T>0$ then
$\Omega_1+\Omega_3^T\Omega_2^{-1}\Omega_3<0$ if and only if $
\left[\begin{array}{ccc}
\Omega_1 & \Omega_3^T\\
* & -\Omega_2
\end{array}\right]<0.$
\end{lemma}
\begin{lemma} \cite{reliable2} \label{uncertain}
Let $M$, $N$ and $F(t)$ be real constant matrices of appropriate dimensions with $F(t)$ satisfying $F^{T}(t)F(t)\leq I$, then there exists a scalar $\epsilon >0$, such that $MF(t)N+(MF(t)N)^{T}\leq \epsilon^{-1}MM^{T}+\epsilon N^{T}N$.
\end{lemma}
\begin{lemma}\label{p1-writ1}\cite{base}
Consider a given matrix $R=R^T>0$. Then, for all continuously differentiable function $\dot{x}$ in $[a,b]\rightarrow \mathbb{R}^n$, the following inequality holds:
\begin{eqnarray*}
-\int_{a}^{b}\dot{x}^T(s)R\dot{x}(s)ds\leq -\frac{1}{b-a}
\begin{bmatrix}
\Pi_1^T \\ \Pi_2^T
\end{bmatrix}^T
\begin{bmatrix}
R & 0 \\ \ast & 3R
\end{bmatrix}
\begin{bmatrix}
\Pi_1^T \\ \Pi_2^T
\end{bmatrix},
\end{eqnarray*}
where $\Pi_1=x(b)-x(a)$ and $\Pi_2=x(b)+x(a)-\frac{2}{b-a}\int_{a}^{b}x(s)ds$.
\end{lemma}
\begin{lemma}\label{p1-writ2}\cite{base}
Suppose that there exists a matrix $M\in \mathbb{R}^{n\times n}$ satisfying
$\begin{bmatrix}
R & M^T \\ \ast & R
\end{bmatrix}\geq 0$
for given symmetric positive definite matrices $\mathbb{R}^{n\times n}$. Then, for any scalar $\theta \in (0, 1)$, the following inequality holds:
\begin{eqnarray*}
\begin{bmatrix}
\frac{1}{\theta}R & 0 \\ 0 & \frac{1}{1-\theta}R
\end{bmatrix}
\geq
\begin{bmatrix}
R & M^T \\ \ast & R
\end{bmatrix}.
\end{eqnarray*}
\end{lemma}
\begin{lemma}\label{jnl}\cite{writ}
For ny constant matrix $W_1>0$, any scalars $a$ and $b$ with $a<b$ and a vector function $x: [a,b]\rightarrow\mathbb{R}^n$, the following integral inequality holds:
\begin{eqnarray*}
\bigg[\int_a^b x(s)ds\bigg]^TW_1\bigg[\int_a^b
x(s)ds\bigg]\leq(b-a)\int_a^b x^T(s)W_1x(s)ds.
\end{eqnarray*}
\end{lemma}
\begin{lemma}\label{wrtl}\cite{writ}
For any matrix $R \in \mathbb{R}^{n \times m} $, $R=R^T > 0$, any differentiable function $\omega$ in $[a,b]\rightarrow \mathbb{R}^n $ the following inequalities holds:
\begin{align*}
\int^{b}_{a} \dot{\omega}^T(s)R\dot{\omega}(s)ds
&\geq \frac{\varsigma^T\bigg[W^T_1RW_1 + \pi^2 W^T_2RW_2 \bigg]\varsigma}{b-a},
\end{align*}
where $\varsigma = [\omega^T(b)\ \omega^T(a)\ \int^b_a \omega^T(s)/(b-a)ds]^T, W_1=[I -I \ 0], W_2=[I/2 \ I/2 \ -I]$.
\end{lemma}
\section{Main results}\label{p1-sec3}
In this section, we obtain the sufficient conditions for stability and stabilization of the singular NCCSs \eqref{p1-eq14} will be derived by using the Lyapunov-Krasovskii functional method and mixed-triggered reliable controller. First, we develop the conditions which ensures that the singular NCCSs \eqref{p1-eq14} is admissible in the absence of external disturbances. Next, this results can be easily extended to obtain mixed-triggered reliable $(\mathcal{Q}, \mathcal{S}, \mathcal{R})$ dissipative controller that guarantees the stability of the system with known and unknown actuator failures by using LMI technique.
\begin{theorem}\label{p1-thm1}
For given positive scalars $\bar{\alpha}$, $\bar{\beta}$, $\lambda$, $\gamma$, $\epsilon_1$, $\epsilon_2$, $\epsilon_3$, $\epsilon_4$ and $\epsilon_f$, the upper bound of time-delays $\zeta_2$, $d_2$, $\tau_2$, $\bar{\theta}$, trigger parameter $\mu$, and given matrices $F$, $\mathcal{Q}=\mathcal{Q}^T$, $\mathcal{R}=\mathcal{R}^T$ and $\mathcal{S}$, the actuator fault matrix $\mathbb{G}$ is known, the singular NCCSs \eqref{p1-eq14} is mean-square asymptotically admissible and strictly dissipative, if there exist symmetric positive definite matrices ${X}_1$, ${X}_2$, $\hat{Q}_{1}$, $\hat{Q}_{2}$, $\hat{Q}_{3}$, $\hat{Q}_{4}$, $\hat{\tilde{Q}}_4$, $\hat{Z}_{1}$, $\hat{Z}_{2}$, $\hat{R}_{1}$, $\hat{R}_{2}$, $\hat{R}_{3}$ $\hat{W}>0$ and matrices $Y_1$, $Y_2$, $\hat{\mathbb{U}}_{i1}$, $\hat{\mathbb{U}}_{i2}$, $\hat{\mathbb{U}}_{i3}$, $\hat{\mathbb{U}}_{i4}$ $(i= 1, 2, 3)$ such that the following LMIs hold
\begin{align}
\label{p1-thm1lmi}X_2^T\mathcal{E}^T=\mathcal{E}X_2&\geq 0,\\
\label{p1-thm1lmi1}\hat{\hat{\Omega}}=
\begin{bmatrix}
[\check{\Omega}]_{22\times 22} & \hat{\hat{\Omega}}_1\\
\ast & \hat{\hat{\Omega}}_2
\end{bmatrix}&<0,\\
\label{p1-thm1lmi2}\begin{bmatrix}
\hat{R}_i & 0 & \hat{\mathbb{U}}_{i1}^T & \hat{\mathbb{U}}_{i3}^T \\
\ast & 3\hat{R}_i & \hat{\mathbb{U}}_{i2}^T  & \hat{\mathbb{U}}_{i4}^T  \\
\ast & \ast &  \hat{R}_i & 0 \\
\ast &  \ast &  \ast &  3\hat{R}_i
\end{bmatrix}&\geq 0, \quad i=1, 2, 3
\end{align}
where $\hat{\hat{\Omega}}$ parameters are given in Appendix~\ref{sec:app:1}.
Furthermore, the desired state feedback reliable
controller gain matrices can be calculated by $\mathcal{K}_1=Y_1X_1^{-1}$ and $\mathcal{K}_2=Y_2X_2^{-1}$.
\end{theorem}
\begin{proof}
The proof of Theorem~\ref{p1-thm1} is referred to Appendix~\ref{sec:app:2}.
\end{proof}

Next, we present the actuator fault matrix $\mathbb{G}$ is unknown and satisfying  \eqref{p1-eq7}, the mixed-triggered reliable controller is designed through the upcoming theorem by utilizing the sufficient conditions in Theorem \ref{p1-thm1}.

\begin{theorem}\label{p1-thm2}
For given positive scalars $\bar{\alpha}$, $\bar{\beta}$, $\lambda$, $\gamma$, $\epsilon_1$, $\epsilon_2$, $\epsilon_3$, $\epsilon_4$ and $\epsilon_f$, the upper bound of time-delays $\zeta_2$, $d_2$, $\tau_2$, $\bar{\theta}$, trigger parameter $\mu$, the actuator fault matrix $\mathbb{G}$ is unknown and matrices $F$, $\mathcal{Q}=\mathcal{Q}^T$, $\mathcal{R}=\mathcal{R}^T$ and $\mathcal{S}$, the singular NCCSs \eqref{p1-eq14} is mean-square asymptotically admissible and strictly $(\mathcal{Q}, \mathcal{S}, \mathcal{R})$ dissipative, if there exist symmetric positive definite matrices ${X}_1$, ${X}_2$, $\hat{Q}_{1}$, $\hat{Q}_{2}$, $\hat{Q}_{3}$, $\hat{Q}_{4}$, $\hat{\tilde{Q}}_4$, $\hat{Z}_{1}$, $\hat{Z}_{2}$, $\hat{R}_{1}$, $\hat{R}_{2}$, $\hat{R}_{3}$ $\hat{W}>0$ and matrices $Y_1$, $Y_2$, $\hat{\mathbb{U}}_{i1}$, $\hat{\mathbb{U}}_{i2}$, $\hat{\mathbb{U}}_{i3}$, $\hat{\mathbb{U}}_{i4}$ $(i= 1, 2, 3)$  such that the following LMI together with \eqref{p1-thm1lmi} and \eqref{p1-thm1lmi2} holds
\begin{align}\label{p1-thm2lmi1}
\hat{\Lambda}=
\begin{bmatrix}
\Theta & \tilde{B} \\
\ast & \tilde{\epsilon}
\end{bmatrix}&<0,
\end{align}
where the parameters in the matrix $\hat{\Lambda}$ are listed in Appendix~\ref{sec:app:3}.
Moreover, the desired primary and secondary controller gain matrices can be obtain as $\mathcal{K}_1=Y_1X_1^{-1}$ and $\mathcal{K}_2=Y_2X_2^{-1}$, respectively.
\end{theorem}
\begin{proof}
We take actuator fault matrix $\mathbb{G}$ is unknown and satisfying the fault constraint \eqref{p1-eq7}, the LMI condition \eqref{p1-thm1lmi1} in Theorem \ref{p1-thm1} for the design of reliable controller can be expressed as
\begin{align}\label{p1-eq40}
\Lambda=\Theta&+
\tilde{B}_1^T\Sigma^T\tilde{Y}_1+\tilde{Y}_1^T\Sigma \tilde{B}_1
+\tilde{B}_2^T\Sigma^T\tilde{Y}_2+\tilde{Y}_2^T\Sigma \tilde{B}_2\nonumber\\&
+\tilde{B}_3^T\Sigma^T\tilde{Y}_3+\tilde{Y}_3^T\Sigma \tilde{B}_3
+\tilde{B}_4^T\Sigma^T\tilde{Y}_4+\tilde{Y}_4^T\Sigma \tilde{B}_4\nonumber\\&
+\tilde{B}_5^T\Sigma^T\tilde{Y}_5+\tilde{Y}_5^T\Sigma\tilde{B}_5
+\tilde{B}_6^T\Sigma^T\tilde{Y}_6+\tilde{Y}_6^T\Sigma \tilde{B}_6\nonumber\\&
+\tilde{B}_7^T\Sigma^T\tilde{Y}_7+\tilde{Y}_7^T\Sigma \tilde{B}_7
+\tilde{B}_8^T\Sigma^T\tilde{Y}_8+\tilde{Y}_8^T\Sigma \tilde{B}_8,
\end{align}
where $\Theta$ is obtained by replacing $\mathbb{G}$ by $\mathbb{G}_0$ in $ \hat{\hat{\Omega}}$.
Thus, it is immediately follows from Lemma \ref{uncertain} that
\begin{align}\label{p1-eq41}
\hat{\Lambda}=\Theta&+
\tilde{\epsilon}_1\tilde{B}_1^T\tilde{B}_1+\tilde{\epsilon}_1^{-1}\tilde{Y}_1^T\tilde{Y}_1
+\tilde{\epsilon}_2\tilde{B}_2^T\tilde{B}_2+\tilde{\epsilon}_2^{-1}\tilde{Y}_2^T\tilde{Y}_2\nonumber\\&
+\tilde{\epsilon}_3\tilde{B}_3^T\tilde{B}_3+\tilde{\epsilon}_3^{-1}\tilde{Y}_3^T\tilde{Y}_3
+\tilde{\epsilon}_4\tilde{B}_4^T\tilde{B}_4+\tilde{\epsilon}_4^{-1}\tilde{Y}_4^T\tilde{Y}_4\nonumber\\&
+\tilde{\epsilon}_5\tilde{B}_5^T\tilde{B}_5+\tilde{\epsilon}_5^{-1}\tilde{Y}_5^T\tilde{Y}_5
+\tilde{\epsilon}_6\tilde{B}_6^T\tilde{B}_6+\tilde{\epsilon}_6^{-1}\tilde{Y}_6^T\tilde{Y}_6\nonumber\\&
+\tilde{\epsilon}_7\tilde{B}_7^T\tilde{B}_7+\tilde{\epsilon}_7^{-1}\tilde{Y}_7^T\tilde{Y}_7
+\tilde{\epsilon}_8\tilde{B}_8^T\tilde{B}_8+\tilde{\epsilon}_8^{-1}\tilde{Y}_8^T\tilde{Y}_8.
\end{align}
Then by using Lemma \ref{p1-schur}, aforementioned condition  \eqref{p1-eq41} is equivalent to LMI \eqref{p1-thm2lmi1}. Hence, the singular NCCSs \eqref{p1-eq14} is mean-square asymptotically stabilized through the proposed controller scheme and strictly $(\mathcal{Q}, \mathcal{S}, \mathcal{R})$ dissipative. This completes the proof of this theorem.
\end{proof}

\section{Numerical example}\label{p1-sec4}
{In this section, we present a numerical example to illustrate the effectiveness of the theoretical results. For this purpose, a power plant boiler-turbine system is considered which can be expressed by $\sum_1$ and $\sum_2$. Then, the system parameter values are borrowed from \cite{NCCS} which are given below:
}
\begin{align*}
\mathcal{E}&=\begin{bmatrix}
1 & 0\\0 & 0
\end{bmatrix},\
A_1=\begin{bmatrix}
-1 & 0\\-1 & -2
\end{bmatrix},\
A_2=
\begin{bmatrix}
1.3 & 1\\0.2 & 0
\end{bmatrix},\\
A_3&=
\begin{bmatrix}
0.2 & 0.1 \\ 0.2 & 1
\end{bmatrix},\
B_1=\begin{bmatrix}
0.2 \\ 0.1
\end{bmatrix},\
B_2=
\begin{bmatrix}
0.2 \\ 1
\end{bmatrix},\
B_3=
\begin{bmatrix}
-0.4 \\0.1
\end{bmatrix},\\
C_1&=\begin{bmatrix}
0 & 0.1
\end{bmatrix},\
C_2=\begin{bmatrix}
-0.3 & 0.1
\end{bmatrix},\
D_1=0.2,\
D_2=0.1.
\end{align*}
{\bf Case 1:} If we set $\bar{\alpha}=0.25$, then the signal is transmitted via mixed-triggered scheme. In this scheme, the rest of parameters involved in this simulations are given as
$\zeta_{2}=0.5$,
$\tau_{2}=0.5$,
$d_{2}=0.5$,
$\bar{\theta}=0.5$,
$\mu=0.16$,
$\lambda=0.05$,
$\gamma=0.1$,
$\epsilon_1=\epsilon_2=\epsilon_3=\epsilon_4=\epsilon_f=1$,
$\epsilon=0.2$,
$h=0.1s$,
$\bar{\beta}=0.02$
and also let we take matrices
$\mathcal{Q}=-0.8$,
$\mathcal{S}=-0.8$,
$\mathcal{R}=1.5$,
$F=\mbox{diag}\{0.02, 0.1\}$.
Now, we look in to the actuator fault matrix $\mathbb{G}$ lie in an interval $[0.6, 0.8]$. Then, by solving the LMIs given in Theorem \ref{p1-thm2}, the corresponding $(\mathcal{Q}, \mathcal{S}, \mathcal{R})$ dissipative control gain matrices can be obtained as
$\mathcal{K}_{1}=10^{-3}\times[ -0.2030  \ \  0.4837]$ and  $\mathcal{K}_{2}= [ -3.8497  \ \ -2.4732 ]$, and event-triggered matrix is
\begin{eqnarray}W=
\begin{bmatrix}
3.9551  &  0.2531\\
    0.2531  &  5.0434
\end{bmatrix}.
\end{eqnarray}
The initial conditions of the primary and secondary plants are given as  $\begin{bmatrix} -5.5 & -2.5\end{bmatrix}^T$ and
$\begin{bmatrix} 6 & -12.96 \end{bmatrix}^T$, respectively.
The external disturbance input $w(t)$ is chosen as
\begin{eqnarray}
w(t)=
\left\lbrace
\begin{array}{l}
\sin(t),\ 0<t\leq 5,\\
0, \mbox{otherwise}
\end{array}.
\right.
\end{eqnarray}
The nonlinear signal of cyber attack is taken as
\begin{eqnarray}f(x_1)=
\begin{bmatrix}
-\tanh(0.02x_{11}(t))\\
-\tanh(0.1x_{12}(t))
\end{bmatrix}.
\end{eqnarray}
The system state and control responses of primary plant are shown in Fig. \ref{p1-fig2}. Also, the state and control responses of secondary plant are presented in Fig. \ref{p1-fig3}. The Bernoulli distribution of random variables $\alpha(t)$ and $\beta(t)$ are plotted in Fig. \ref{p1-fig4} which are introduced to connecting the mixed triggered scheme by switch rule and occurrence of cyber attack used in the simulation. In Fig. \ref{p1-fig7}, the simulation curves of  attack function is depicted.
\begin{figure}[!htb]
	\centering
	\begin{minipage}{\textwidth}
		\begin{minipage}[b]{.1\textwidth}
			\includegraphics[width=4.5cm, height=3.5cm]{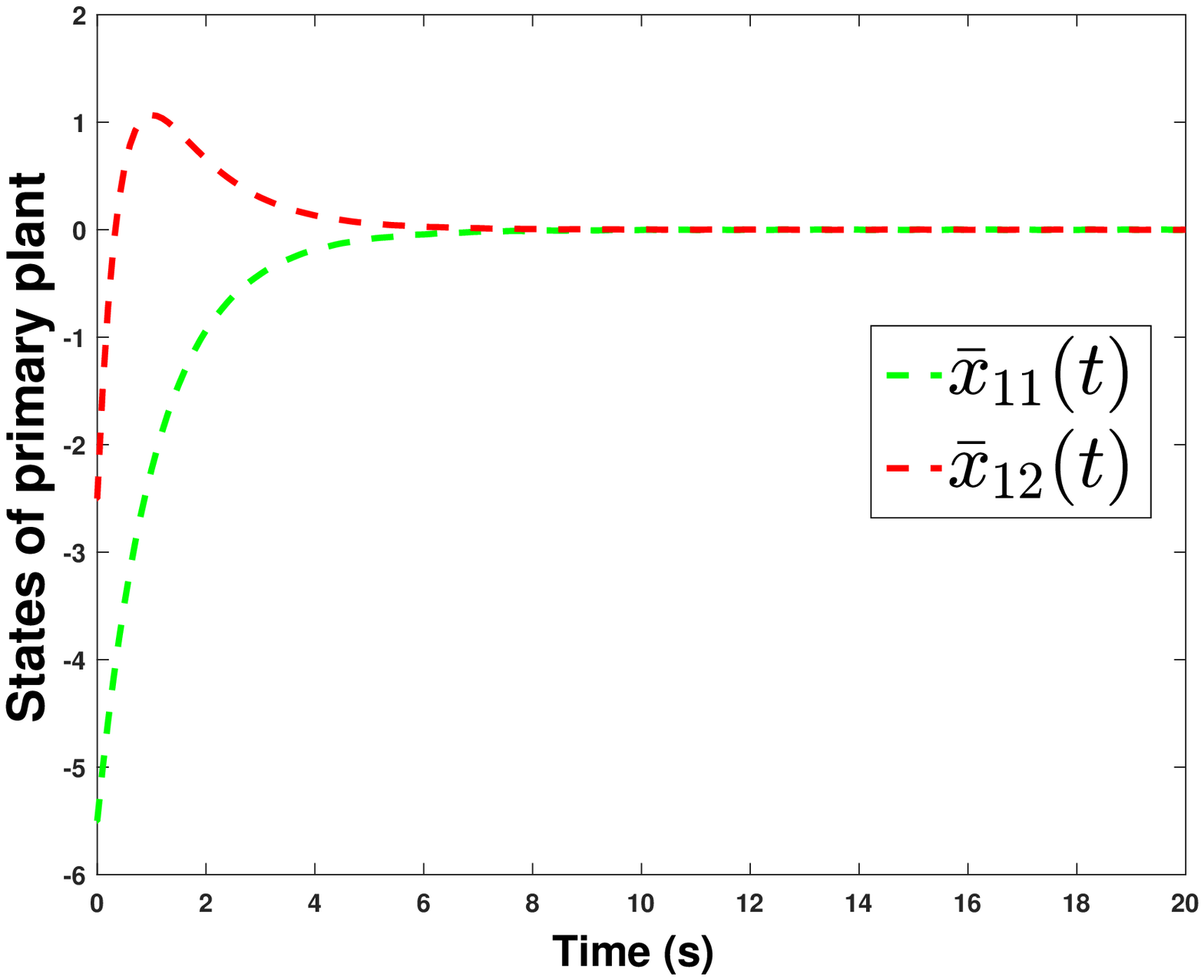}
			\centering
		\end{minipage}
		\begin{minipage}[b]{.5\textwidth}
			\includegraphics[width=4.5cm, height=3.5cm]{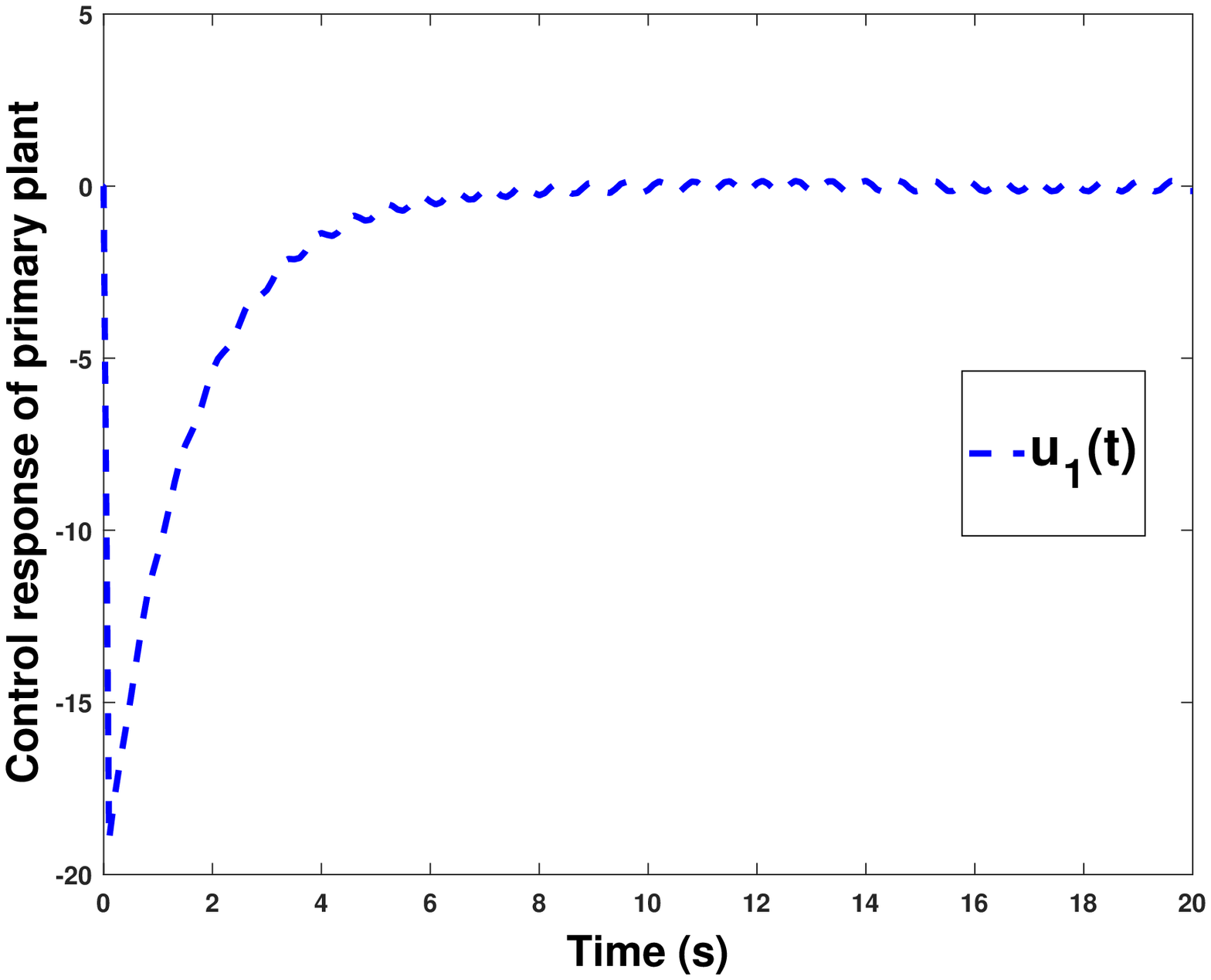}
			\centering
		\end{minipage}
	\end{minipage}
		\caption{State and control responses for primary plant}\label{p1-fig2}
\end{figure}
\begin{figure}[!htb]
	\centering
	\begin{minipage}{\textwidth}
		\begin{minipage}[b]{.1\textwidth}
			\includegraphics[width=4.5cm, height=3.5cm]{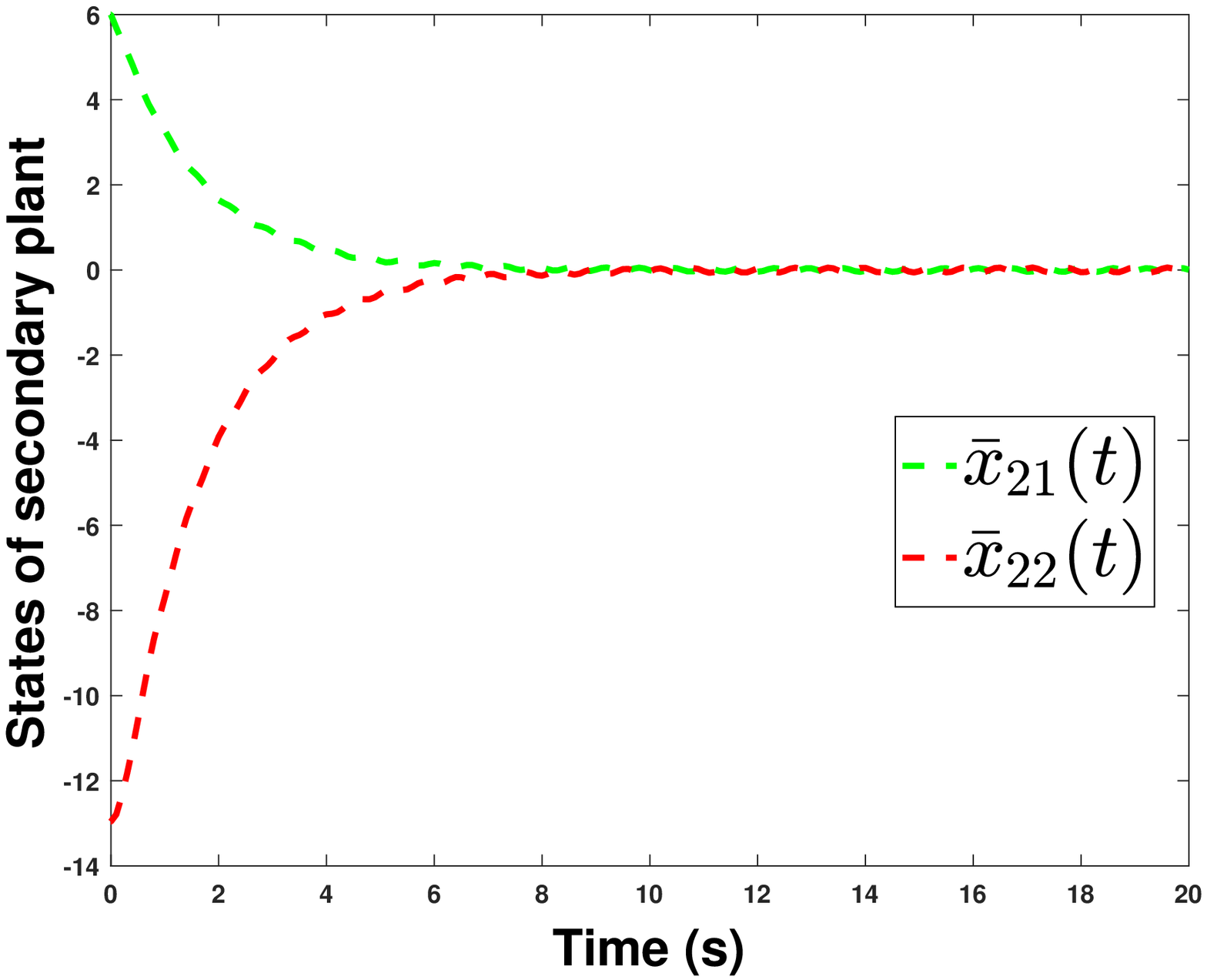}
			\centering
		\end{minipage}
		\begin{minipage}[b]{.5\textwidth}
			\includegraphics[width=4.5cm, height=3.5cm]{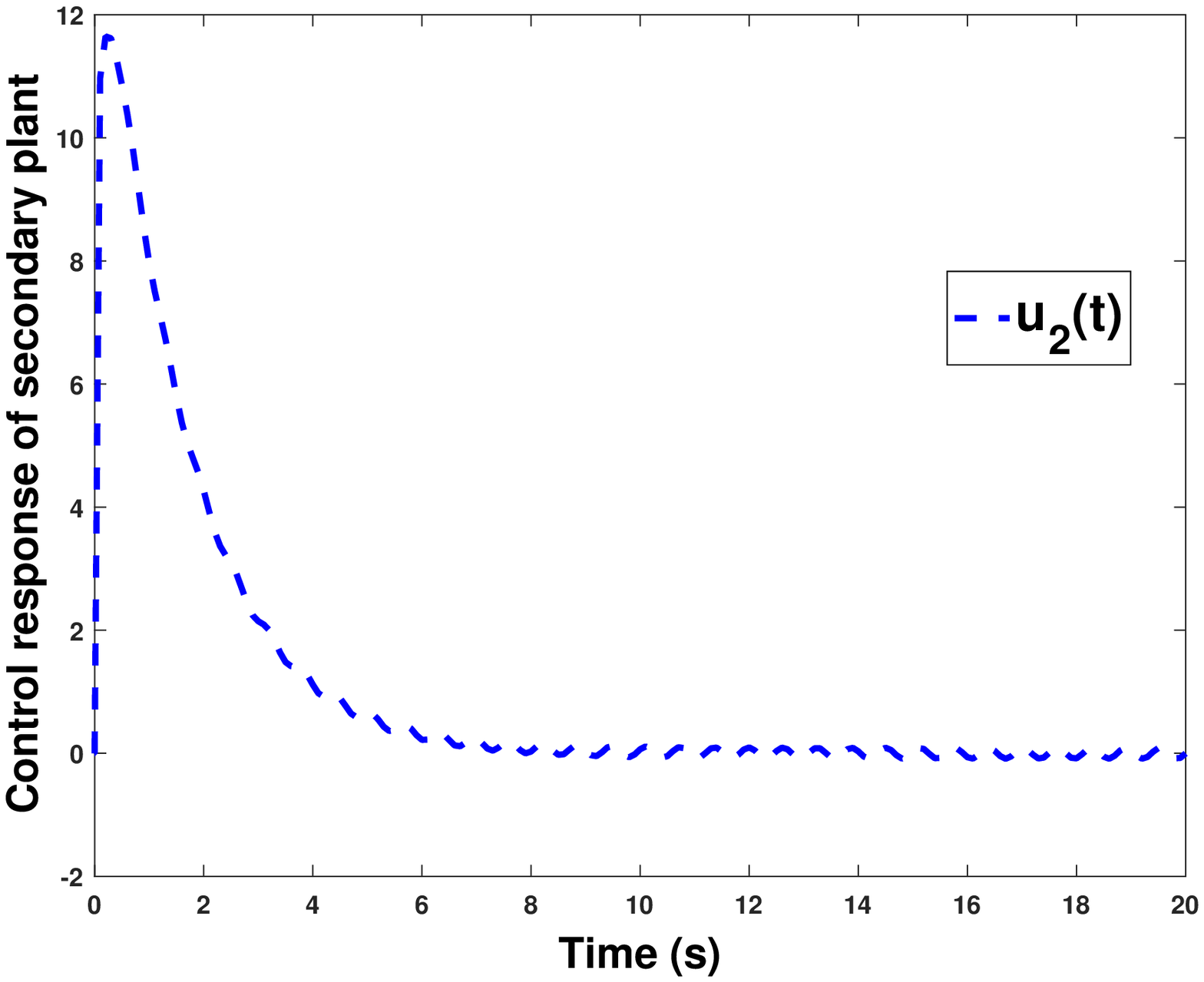}
			\centering
		\end{minipage}
	\end{minipage}
		\caption{ State and control responses for secondary plant}\label{p1-fig3}
\end{figure}
\begin{figure}[!htb]
	\centering
	\begin{minipage}{\textwidth}
		\begin{minipage}[b]{.1\textwidth}
			\includegraphics[width=4.5cm, height=3.5cm]{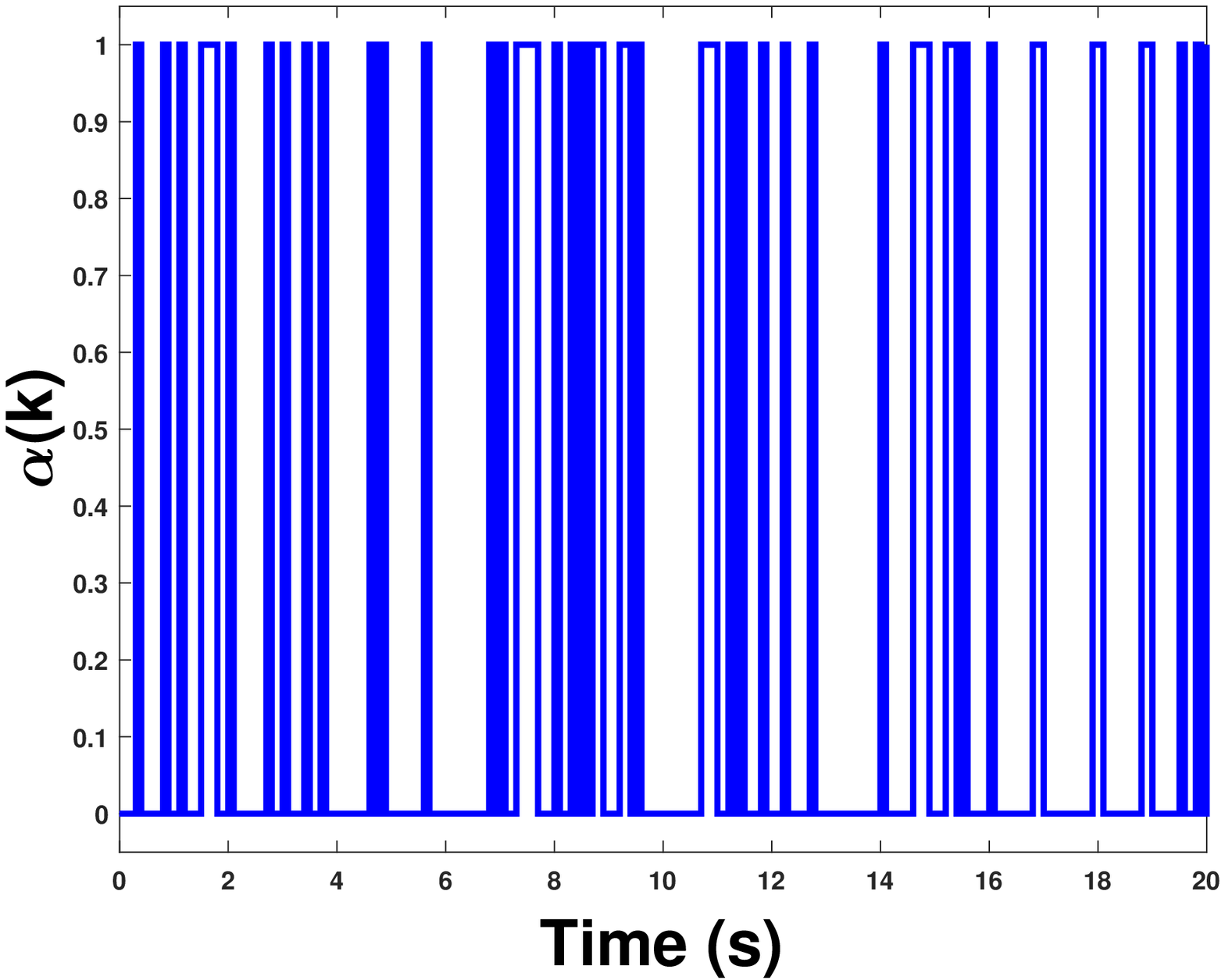}
			\centering
		\end{minipage}
		\begin{minipage}[b]{.5\textwidth}
			\includegraphics[width=4.5cm, height=3.5cm]{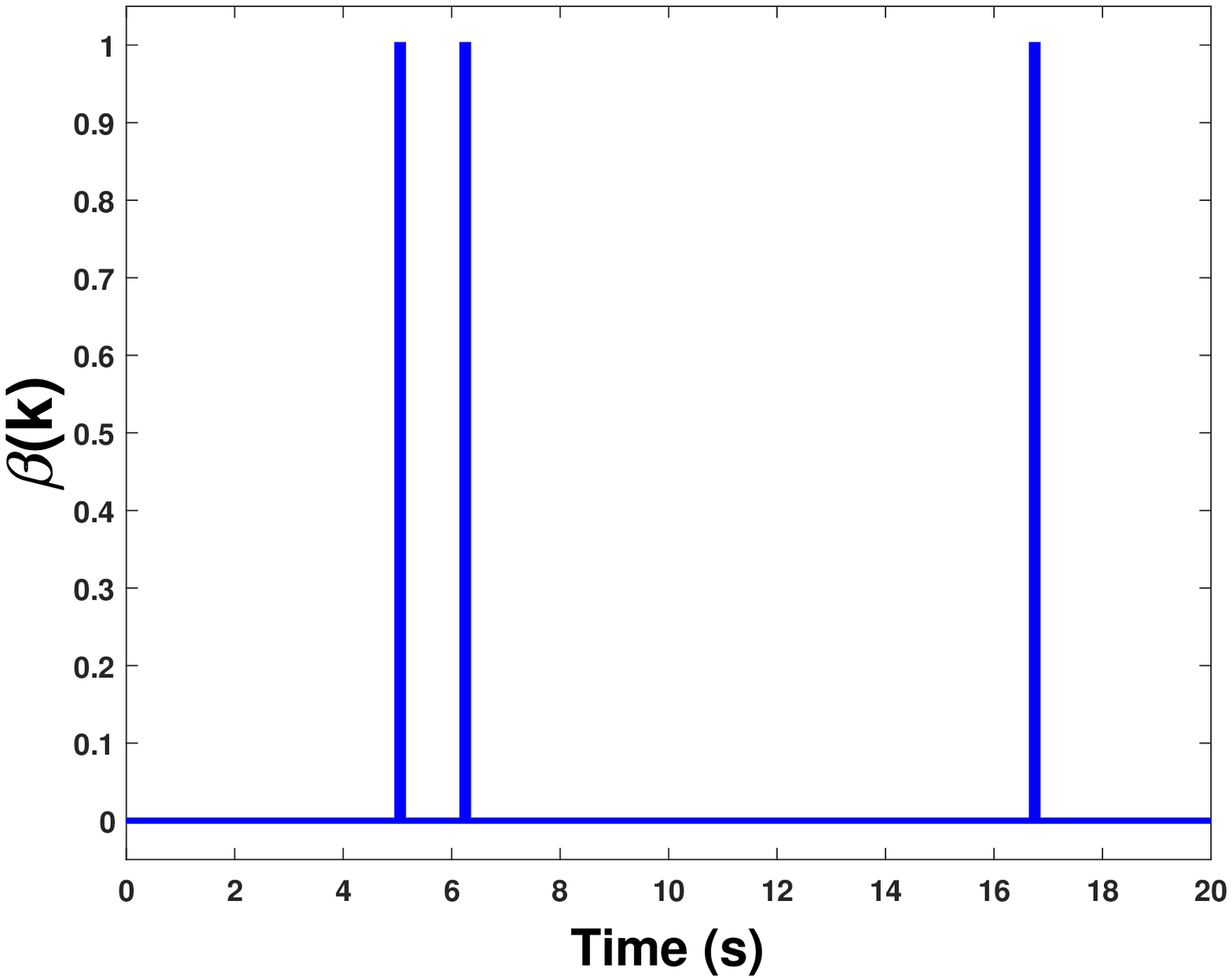}
			\centering
		\end{minipage}
	\end{minipage}
		\caption{$\alpha(t)$ with $\bar{\alpha}=0.25$ and $\beta(t)$ with $\bar{\beta}=0.02$}\label{p1-fig4}
\end{figure}
\begin{figure}[!htb]
	\centering
		\begin{minipage}[b]{.4\textwidth}
			\includegraphics[width=4.5cm, height=3.5cm]{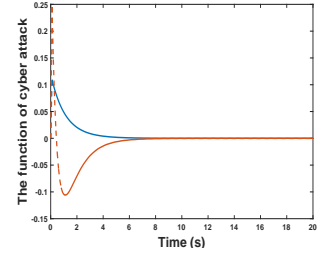}
			\centering
		\end{minipage}
		\caption{Attack function }\label{p1-fig7}
\end{figure}
\begin{figure}[!htb]
	\centering
	\begin{minipage}{\textwidth}
		\begin{minipage}[b]{.1\textwidth}
			\includegraphics[width=4.5cm, height=3.5cm]{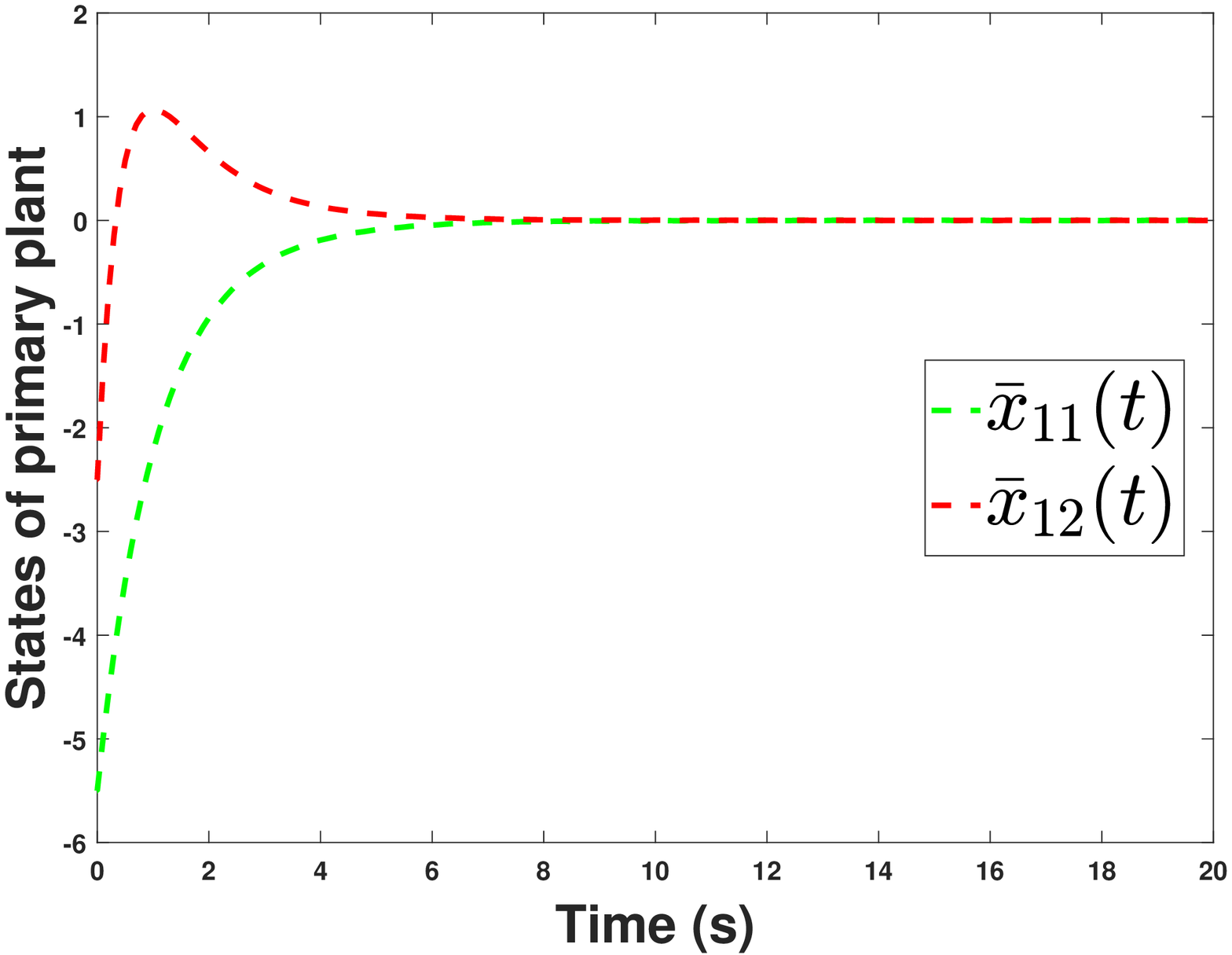}
			\centering
		\end{minipage}
		\begin{minipage}[b]{.5\textwidth}
			\includegraphics[width=4.5cm, height=3.5cm]{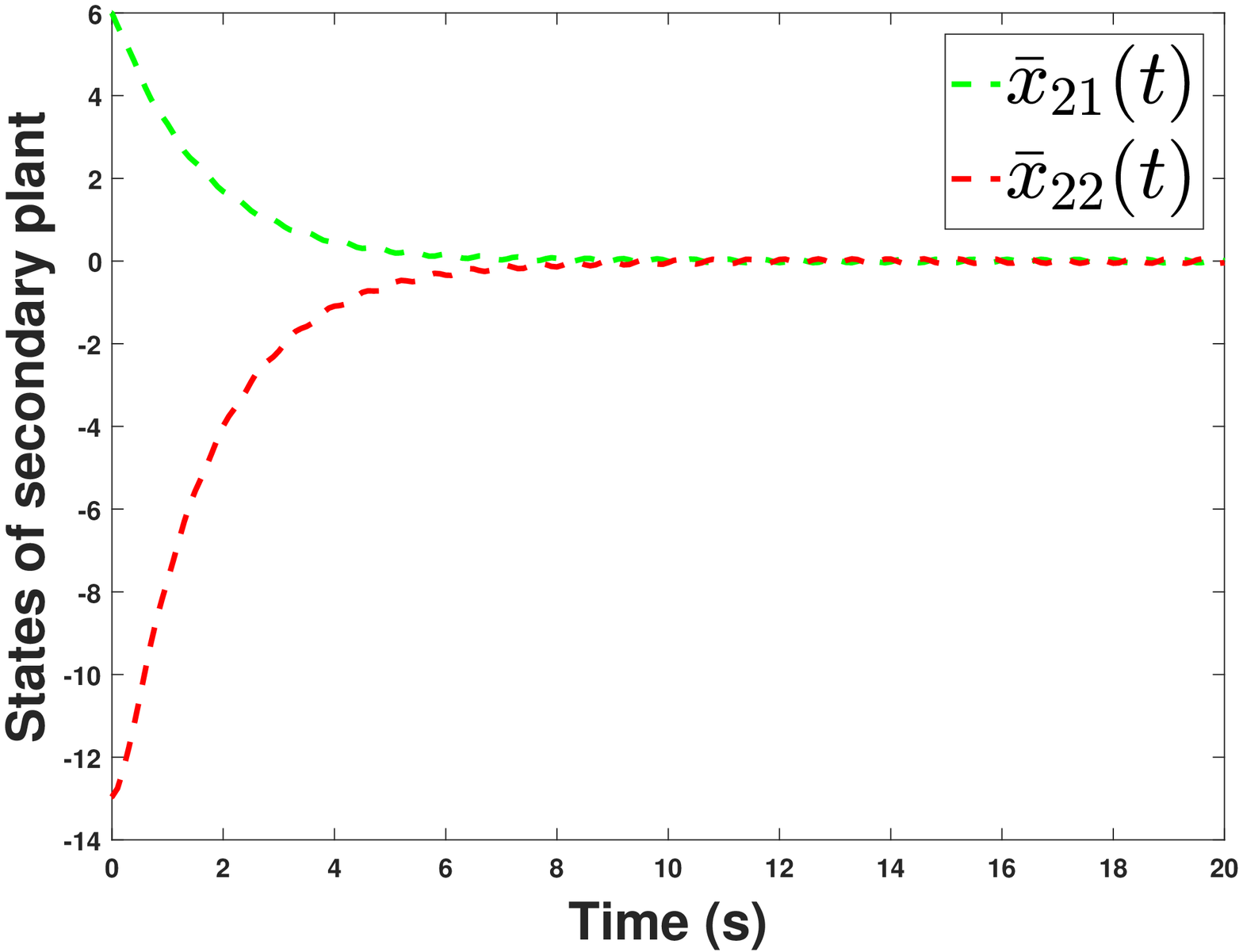}
			\centering
		\end{minipage}
	\end{minipage}
		\caption{State responses of primary and secondary plants when $\bar{\alpha}=1$}\label{p1-fig5}
\end{figure}
\begin{figure}[!htb]
	\centering
	\begin{minipage}{\textwidth}
		\begin{minipage}[b]{.1\textwidth}
			\includegraphics[width=4.5cm, height=3.5cm]{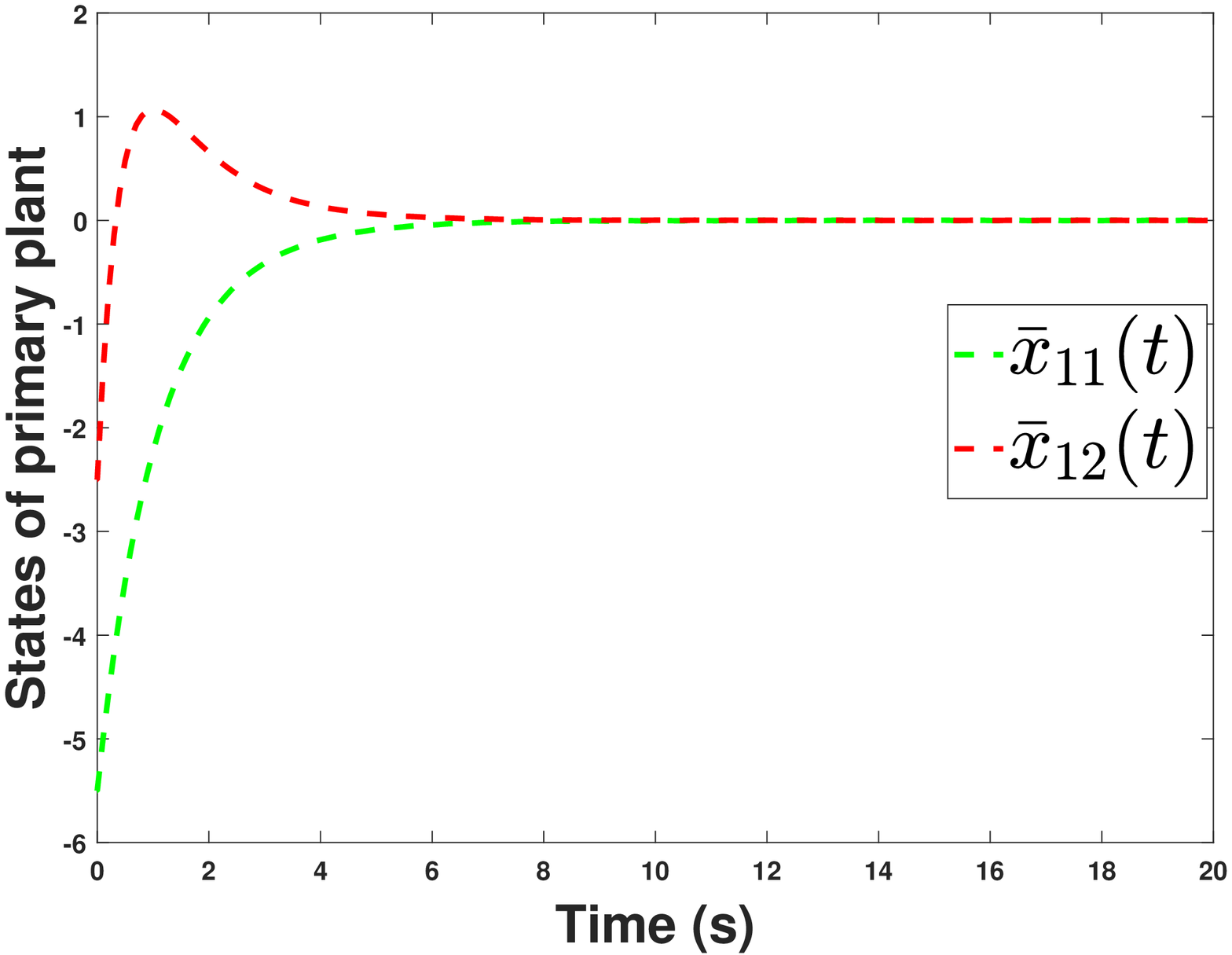}
			\centering
		\end{minipage}
		\begin{minipage}[b]{.5\textwidth}
			\includegraphics[width=4.5cm, height=3.5cm]{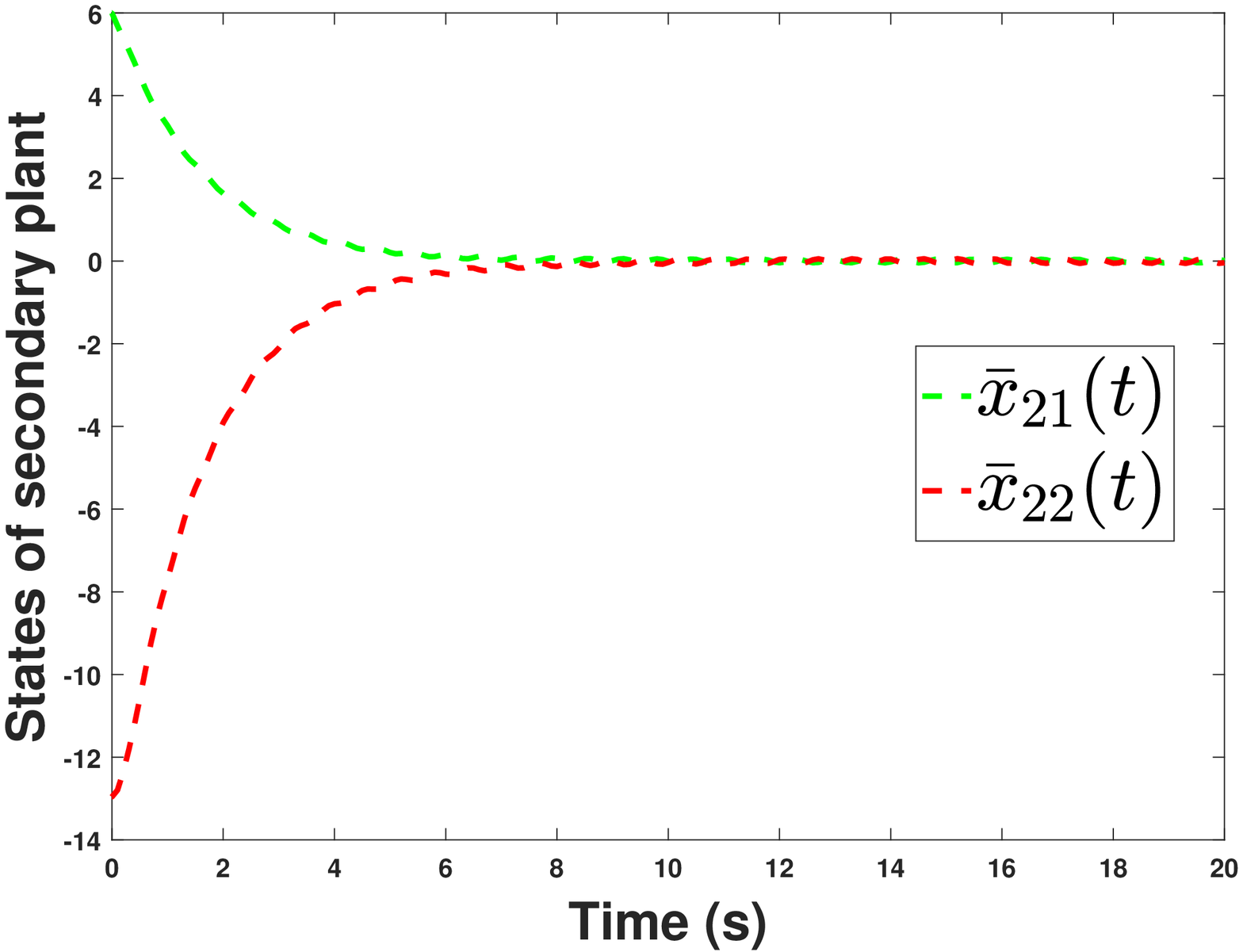}
			\centering
		\end{minipage}
	\end{minipage}
		\caption{State responses of primary and secondary plants when $\bar{\alpha}=0$}\label{p1-fig6}
\end{figure}
\begin{figure}[!htb]
	\centering
	\begin{minipage}{\textwidth}
		\begin{minipage}[b]{.1\textwidth}
			\includegraphics[width=4.5cm, height=3.5cm]{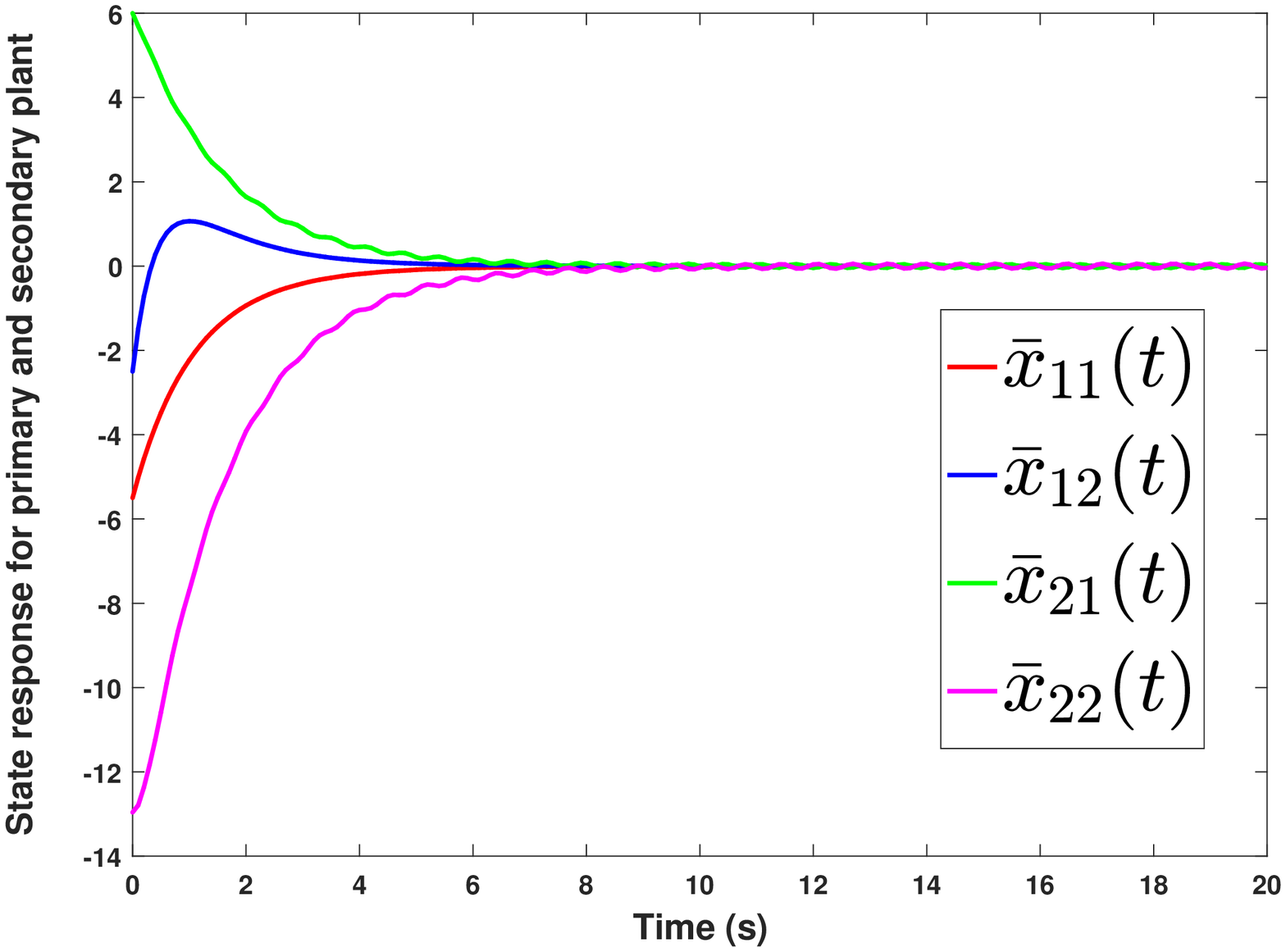}
			\centering
		\end{minipage}
		\begin{minipage}[b]{.5\textwidth}
			\includegraphics[width=4.5cm, height=3.5cm]{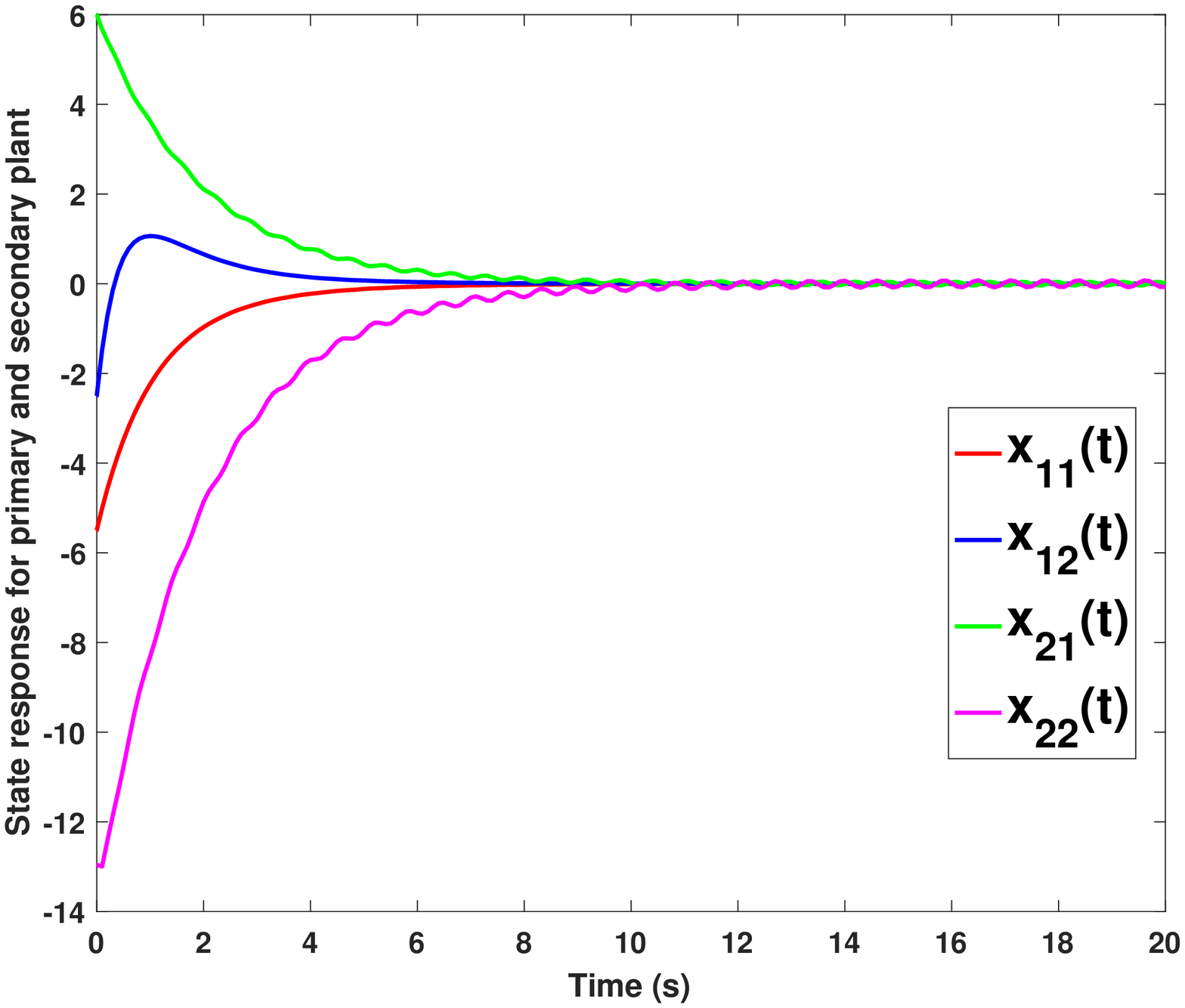}
			\centering
		\end{minipage}
	\end{minipage}
		\caption{State response of proposed and Figure 4 in \cite{NCCS}}\label{p1-fig8}
\end{figure}

{\bf Case 2:} By letting $\bar{\alpha}=1$, then the signal is transmitted via time-triggered scheme. choose sampling period $h=0.01s$.
Moreover, the initial condition, exogenous disturbance signal,
and the rest of parameters are taken as same as in the previous
case. From Theorem \ref{p1-thm2}, we calculate the parameters of the controller gain matrices as
$\mathcal{K}_{1}=10^{-3}\times[ -0.2103 \ \ 0.1165]$ and  $\mathcal{K}_{2}= [-3.7951 \ \ -2.4750]$. The state responses of primary plant and secondary plant are shown in Fig. \ref{p1-fig5}.

{\bf Case 3:} If we fix $\bar{\alpha}=0$, then the signal is transmitted via event-triggered scheme.
According to Theorem \ref{p1-thm2}, the state feed-back controller gains are achieved as
$\mathcal{K}_{1}=10^{-3}\times[  -0.1581  \ \  0.1355]$ and  $\mathcal{K}_{2}= [ -3.8146   -2.4847]$, and the corresponding event-triggered matrix is
\begin{eqnarray}
W=
\begin{bmatrix}
3.3959  &  0.2204\\
    0.2204  &  4.3544
\end{bmatrix}.
\end{eqnarray}
The state responses of primary plant and secondary plant are shown in Fig. \ref{p1-fig6}.
From the simulation results, the state response of the singular NCCSs under the proposed controller scheme and the $H_\infty$ controller scheme are given in Fig. \ref{p1-fig8}. We easily conclude from these figures that the system trajectories converges quickly to the equilibrium point under the proposed controller then the $H_\infty$ controller in \cite{NCCS} which shows the superiority of the proposed controller scheme.
\section{CONCLUSIONS}\label{p1-sec5}
In this paper, the mixed-triggered reliable control problem for singular networked cascade control systems with actuator saturation and randomly occurring cyber attack has been studied. In particular, a mixed-triggered scheme is introduced to reduce the burden of network bandwidth which is modeled in a probabilistic way by using Bernoulli distributed random variable to describe the switching rule connecting time and event-triggered. With the help of LMI technique and Lyapunov-Krasovskii functional, a set of sufficient conditions has been obtained for guaranteeing the closed-loop singular NCCSs can achieve the desired results. At last, the power plant-boiler-turbine system is employed to illustrate the effectiveness of the proposed method.

\section{Appendix}
\subsection{Parameters in Theorem~\ref{p1-thm1}}\label{sec:app:1}
The parameters in the matrix $\check\Omega$ are
$\check{\Omega}_{1,1}=2A_1X_1+\hat{Q}_1+\hat{Q}_2+\hat{Q}_3-4\hat{R}_1-4\hat{R}_2-4\hat{R}_3$,
$\check{\Omega}_{1,2}=-2\hat{R}_1-(\hat{M}_1+\hat{M}_2+\hat{M}_3+\hat{M}_4)^T$,
$\check{\Omega}_{1,3}=(\hat{M}_1+\hat{M}_2-\hat{M}_3-\hat{M}_4)^T$,
$\check{\Omega}_{1,4}=6\hat{R}_1$,
$\check{\Omega}_{1,5}=2(\hat{M}_3+\hat{M}_4)^T$,
$\check{\Omega}_{1,6}=-2\hat{R}_2-(\hat{N}_1+\hat{N}_2+\hat{N}_3+\hat{N}_4)^T$,
$\check{\Omega}_{1,7}=(\hat{N}_1+\hat{N}_2-\hat{N}_3-\hat{N}_4)^T$,
$\check{\Omega}_{1,8}=6\hat{R}_2$,
$\check{\Omega}_{1,9}=2(\hat{N}_3+\hat{N}_4)^T$,
$\check{\Omega}_{1,10}=-2\hat{R}_3-(\hat{S}_1+\hat{S}_2+\hat{S}_3+\hat{S}_4)^T$,
$\check{\Omega}_{1,11}=(\hat{S}_1+\hat{S}_2-\hat{S}_3-\hat{S}_4)^T$,
$\check{\Omega}_{1,12}=6\hat{R}_3$,
$\check{\Omega}_{1,13}=2(\hat{S}_3+\hat{S}_4)^T$,
$\check{\Omega}_{1,14}=B_1C_2X_2$,
$\check{\Omega}_{1,22}=B_1D_2-X_1C_1^T\mathcal{S}$,
$\check{\Omega}_{2,2}=-8\hat{R}_1+2(\hat{M}_1-\hat{M}_2+\hat{M}_3-\hat{M}_4)^T$,
$\check{\Omega}_{2,3}=-2\hat{R}_1+(-\hat{M}_1+\hat{M}_2+\hat{M}_3-\hat{M}_4)^T$,
$\check{\Omega}_{2,4}=6\hat{R}_1+2(\hat{M}_2+\hat{M}_4)^T$,
$\check{\Omega}_{2,5}=6\hat{R}_1-2\hat{M}_3^T+2\hat{M}_4^T$,
$\check{\Omega}_{2,14}=(1-\bar{\beta})\bar{\alpha}Y_1^T\mathbb{G}^TB_2^T$,
$\check{\Omega}_{3,3}=-4\hat{R}_1-\hat{Q}_1$,
$\check{\Omega}_{3,4}=-2\hat{M}_2^T+2\hat{M}_4^T$,
$\check{\Omega}_{3,5}=6\hat{R}_1$,
$\check{\Omega}_{4,4}=-12\hat{R}_1$,
$\check{\Omega}_{4,5}=-4\hat{M}_4^T$,
$\check{\Omega}_{5,5}=-12\hat{R}_1$,
$\check{\Omega}_{6,6}=-8\hat{R}_2+2(\hat{N}_1-\hat{N}_2+\hat{N}_3-\hat{N}_4)^T+\mu \hat{W}$,
$\check{\Omega}_{6,7}=-2\hat{R}_2+(-\hat{N}_1+\hat{N}_2+\hat{N}_3-\hat{N}_4)^T$,
$\check{\Omega}_{6,8}=6\hat{R}_2+2\hat{N}_2^T+2\hat{N}_4^T$,
$\check{\Omega}_{6,9}=6\hat{R}_2-2\hat{N}_3^T+2\hat{N}_4^T$,
$\check{\Omega}_{6,14}=(1-\bar{\beta})(1-\bar{\alpha})Y_1^T\mathbb{G}^TB_2^T$,
$\check{\Omega}_{7,7}=-4\hat{R}_2-\hat{Q}_2$,
$\check{\Omega}_{7,8}=-2\hat{N}_2^T+2\hat{N}_4^T$,
$\check{\Omega}_{7,9}=6\hat{R}_2$,
$\check{\Omega}_{8,8}=-12\hat{R}_2$,
$\check{\Omega}_{8,9}=-4\hat{N}_4^T$,
$\check{\Omega}_{9,9}=-12\hat{R}_2$,
$\check{\Omega}_{10,10}=-8\hat{R}_3+2(\hat{S}_1-\hat{S}_2+\hat{S}_3-\hat{S}_4)^T$,
$\check{\Omega}_{10,11}=-2\hat{R}_3+(-\hat{S}_1+\hat{S}_2+\hat{S}_3-\hat{S}_4)^T$,
$\check{\Omega}_{10,12}=6\hat{R}_3+2\hat{S}_2^T+2\hat{S}_4^T$,
$\check{\Omega}_{10,13}=6\hat{R}_3-2\hat{S}_3^T+2\hat{S}_4^T$,
$\check{\Omega}_{11,11}=-4\hat{R}_3-\hat{Q}_3$,
$\check{\Omega}_{11,12}=-2\hat{S}_2^T+2\hat{S}_4^T$,
$\check{\Omega}_{11,13}=6\hat{R}_3$,
$\check{\Omega}_{12,12}=-12\hat{R}_3$,
$\check{\Omega}_{12,13}=-4\hat{S}_4^T$,
$\check{\Omega}_{13,13}=-12\hat{R}_3$,
$\check{\Omega}_{14,14}=\hat{Q}_4+\hat{\tilde{Q}}_4+2A_2X_2+2B_2Y_2-\mathcal{E}^T\hat{Z}_2\mathcal{E}-\frac{\pi^2}{4}\mathcal{E}^T\hat{Z}_2\mathcal{E}$,
$\check{\Omega}_{14,15}=A_3X_2+\mathcal{E}^T\hat{Z}_2\mathcal{E}-\frac{\pi^2}{4}\mathcal{E}^T\hat{Z}_2\mathcal{E}$,
$\check{\Omega}_{14,17}=\frac{\pi^2}{4}\mathcal{E}^T\hat{Z}_2\mathcal{E}$,
$\check{\Omega}_{14,19}=(1-\bar{\beta})(1-\bar{\alpha})B_2\mathbb{G}Y_1$,
$\check{\Omega}_{14,20}=\bar{\beta}B_2\mathbb{G}Y_1$,
$\check{\Omega}_{14,21}=-B_2$,
$\check{\Omega}_{14,22}=B_3$,
$\check{\Omega}_{15,15}=-2\mathcal{E}^T\hat{Z}_2\mathcal{E}-2\frac{\pi^2}{4}\mathcal{E}^T\hat{Z}_2\mathcal{E}-(1-\lambda)\hat{\tilde{Q}}_4$,
$\check{\Omega}_{15,16}=\mathcal{E}^T\hat{Z}_2\mathcal{E}-\frac{\pi^2}{4}\mathcal{E}^T\hat{Z}_2\mathcal{E}$,
$\check{\Omega}_{15,17}=\frac{\pi^2}{4}\mathcal{E}^T\hat{Z}_2\mathcal{E}$,
$\check{\Omega}_{15,18}=\frac{\pi^2}{4}\mathcal{E}^T\hat{Z}_2\mathcal{E}$,
$\check{\Omega}_{16,16}=-\hat{Q}_4-\mathcal{E}^T\hat{Z}_2\mathcal{E}-\frac{\pi^2}{4}\mathcal{E}^T\hat{Z}_2\mathcal{E}$,
$\check{\Omega}_{16,18}=\frac{\pi^2}{4}\mathcal{E}^T\hat{Z}_2\mathcal{E}$,
$\check{\Omega}_{17,17}=-\pi^2\mathcal{E}^T\hat{Z}_2\mathcal{E}-\frac{1}{\bar{\theta}}\hat{Z}_1$,
$\check{\Omega}_{18,18}=-\pi^2\mathcal{E}^T\hat{Z}_2\mathcal{E}-\frac{1}{\bar{\theta}}\hat{Z}_1$,
$\check{\Omega}_{19,19}=-\hat{W}I$,
$\check{\Omega}_{20,20}=-\bar{\beta}I$,
$\check{\Omega}_{21,21}=-I$,
$\check{\Omega}_{22,22}=-D_1^T\mathcal{S}-\mathcal{R}+\gamma I$,
$\hat{\mathbb{U}}_{1}=\hat{M}$,
$\hat{\mathbb{U}}_{2}=\hat{N}$,
$\hat{\mathbb{U}}_{3}=\hat{S}$,\
$\hat{\hat{\Omega}}_1=
\Big[\zeta_2\hat{\Omega}_1^T \quad
d_2\hat{\Omega}_1^T \quad
\tau_2\hat{\Omega}_1^T \ \
\theta_2\hat{\Omega}_2^T \ \
\theta_2\hat{\Omega}_3^T \ \
\theta_2\hat{\Omega}_4^T \ \
\theta_2\hat{\Omega}_5^T \\
\sqrt{\epsilon}\hat{\Omega}_6^T \ \
\sqrt{\epsilon}\sigma\hat{\Omega}_7^T \ \
\sqrt{\epsilon}\delta\hat{\Omega}_8^T \ \
\sqrt{\epsilon}\sigma\delta\hat{\Omega}_9^T\ \
\hat{\Omega}_{10}^T \ \
\hat{\Omega}_{11}^T\Big]$,
$\hat{\hat{\Omega}}_2=\mbox{diag}\big\{-\tilde{\kappa}_1, -\tilde{\kappa}_2, -\tilde{\kappa}_3, -\tilde{\kappa}_4, -\tilde{\kappa}_4, -\tilde{\kappa}_4, -\tilde{\kappa}_4, -I, -I, -I,\\ -I, -I, -I\big\}$,\ \
$\hat{\Omega}_1=
\big[
A_1X_1^T \ 0_{12n} \ B_1C_2X_2^T \ 0_{7n} \ B_1D_2
\big]$,\
$\hat{\Omega}_2=
\big[
0 \ \bar{\alpha}\bar{\beta}_1\bar{\Upsilon}_1 \ 0_{3n} \ \bar{\alpha}_1\bar{\beta}_1\bar{\Upsilon}_1 \ 0_{7n} \  A_2X_2^T+\bar{\Upsilon}_2 \ A_3X_2^T \ 0_{3n} \  \bar{\alpha}_1\bar{\beta}_1\bar{\Upsilon}_1 \ \bar{\beta}\bar{\Upsilon}_1 \ -B_2 \ B_3
\big]$,\
$\hat{\Omega}_3=
\big[
0 \ \ \sigma\bar{\beta}_1\bar{\Upsilon}_1 \ \ 0_{3n} \ \ \sigma\bar{\beta}_1\bar{\Upsilon}_1 \ \ 0_{12n} \ \ \sigma\bar{\beta}_1\bar{\Upsilon}_1 \ \  0_{3n}
\big]$,\
$\hat{\Omega}_4=
\big[
0 \ \ -\delta\bar{\alpha}_1\bar{\Upsilon}_1 \ \ 0_{3n} \ \ -\delta\bar{\alpha}_1\bar{\Upsilon}_1  \ \ 0_{12n} \ \ -\delta\bar{\alpha}_1\bar{\Upsilon}_1  \ \ \delta\bar{\alpha}_1\bar{\Upsilon}_1 \ \ 0_{2n}
\big]$,\
$\hat{\Omega}_5=
\big[
0 \ \ \sigma\delta \bar{\Upsilon}_1 \ \ 0_{3n} \ \ -\sigma\delta \bar{\Upsilon}_1 \ \ 0_{12n} \ \ -\sigma\delta \bar{\Upsilon}_1 \ \  0_{3n}
\big]$,\
$\hat{\Omega}_6=
\big[
0 \ \bar{\beta}_1\bar{\alpha}\mathbb{G}Y_1 \ 0_{3n} \ \bar{\beta}_1\bar{\alpha}_1\mathbb{G}Y_1 \ 0_{7n} \ K_2 \ 0_{4n} \ \bar{\beta}_1\bar{\alpha}_1\mathbb{G}Y_1
\ \bar{\beta}\mathbb{G}Y_1 \ 0_{2n}
\big]$,
$\hat{\Omega}_7=
\big[
0 \ \ \bar{\beta}_1\mathbb{G}Y_1 \ \ 0_{3n} \ \ -\bar{\beta}_1\mathbb{G}Y_1 \ \ 0_{12n} \ \ -\bar{\beta}_1\mathbb{G}Y_1  \ \ 0_{3n}
\big]$,
$\hat{\Omega}_8=
\big[
0 \ \ -\bar{\alpha}\mathbb{G}Y_1 \ \ 0_{3n} \ \ -\bar{\alpha}_1\mathbb{G}Y_1 \ \ 0_{12n} \ \ -\bar{\alpha}_1\mathbb{G}Y_1 \ \ \mathbb{G}Y_1 \ \ 0_{2n}
\big]$,\
$\hat{\Omega}_9=
\big[
0 \ \ -\mathbb{G}Y_1 \ \ 0_{3n} \ \ \mathbb{G}Y_1 \ \ 0_{12n} \ \ \mathbb{G}Y_1 \ \  0_{3n}
\big]$,
$\hat{\Omega}_{10}=
\big[
0_{9n} \ \ \sqrt{\bar{\beta}}FX_1^T \ \ 0_{12n}
\big]$,\
$\hat{\Omega}_{11}=
\big[
\sqrt{\mathcal{Q}}C_1X_1^T \ \ 0_{20n} \ \ \sqrt{\mathcal{Q}}D_1
\big]$,\
$\bar{\Upsilon}_1=B_2\mathbb{G}Y_1$, $\bar{\Upsilon}_2=B_2Y_2$,\
$\tilde{\kappa}_1=-2\epsilon_1X_1+\epsilon_1^2\hat{R}_1$,
$\tilde{\kappa}_2=-2\epsilon_2X_1+\epsilon_2^2\hat{R}_2$,
$\tilde{\kappa}_3=-2\epsilon_3X_1+\epsilon_3^2\hat{R}_3$,
$\tilde{\kappa}_4=-2\epsilon_1X_2+\epsilon_4^2\hat{Z}_2$,
and the rest of parameters are zero.

\subsection{Proof of Theorem~\ref{p1-thm1}}\label{sec:app:2}
\begin{proof}
In order to prove that the nominal system \eqref{p1-eq14} with $w=0$ is admissible, we first prove that the system \eqref{p1-eq14} is mean-square asymptotically stable. For this purpose, we design a Lyapunov-Krasovskii functional candidate in the following form:
$V=\sum_{i=1}^{4}V_i$,
where
{\begin{align*}
V_1=& \bar{x}_1^TP_1\bar{x}_1+\bar{x}_2^T\mathcal{E}^TP_2\bar{x}_2, \nonumber \\
V_2=&\int_{t-\zeta_2}^{t}\bar{x}_1^T(s)Q_1\bar{x}_1(s)ds +\int_{t-d_2}^{t}\bar{x}_1^T(s)Q_2\bar{x}_1(s)ds\\&
+\int_{t-\tau_2}^{t}\bar{x}_1^T(s)Q_3\bar{x}_1(s)ds
+\int_{t-\theta(t)}^{t}\bar{x}_2^T(s)\tilde{Q}_4\bar{x}_2(s)ds\\&
+\int_{t-\bar{\theta}}^{t}\bar{x}_2^T(s)Q_4\bar{x}_2(s)ds, \nonumber \\
V_3=&\zeta_2\int_{-\zeta_2}^{0}\int_{t+\sigma}^{t}\dot{\bar{x}}_1^T(s)R_1\dot{\bar{x}}_1(s)dsd\sigma
\\&+d_2\int_{-d_2}^{0}\int_{t+\sigma}^{t}\dot{\bar{x}}_1^T(s)R_2\dot{\bar{x}}_1(s)dsd\sigma\\&
+\tau_2\int_{-\tau_2}^{0}\int_{t+\sigma}^{t}\dot{\bar{x}}_1^T(s)R_3\dot{\bar{x}}_1(s)dsd\sigma,\nonumber\\
V_4=&
\int_{-\bar{\theta}}^{0}\int_{t+\sigma}^{t}\bar{x}_2^T(s)Z_1\bar{x}_2(s)dsd\sigma \\&+
\bar{\theta}\int_{-\bar{\theta}}^{0}\int_{t+\sigma}^{t}\dot{\bar{x}}_2^T(s)\mathcal{E}^TZ_2\mathcal{E}\dot{\bar{x}}_2(s)dsd\sigma.
\end{align*}}
Next, we evaluate the time derivatives  $\dot{V}$ along the  trajectories of the closed-loop singular NCCSs \eqref{p1-eq14} and taking the mathematical expectation, we have
{\begin{align}
\label{lkf1}\mathbb{E}[\dot{V}_1]&=\mathbb{E}[2 \bar{x}_{1}^TP_{1}\dot{\bar{x}}_1+2\bar{x}_{2}^TP_{2}\mathcal{E}\dot{\bar{x}}_2],\\
\label{lkf2}\mathbb{E}[\dot V_2]&= \mathbb{E}[\bar{x}_1^T(Q_{1}+Q_{2}+Q_{3})\bar{x}_1
- \bar{x}_1^T(t-\zeta_2)Q_{1}\bar{x}_1(t-\zeta_2)\nonumber\\&
-\bar{x}_1^T(t-d_2)Q_{2}\bar{x}_1(t-d_2)
-\bar{x}_1^T(t-\tau_2)Q_{3}\bar{x}_1(t-\tau_2)\nonumber\\&
+\bar{x}_2^T(t)(\tilde{Q}_{4}+{Q}_{4})\bar{x}_2(t)
-\bar{x}_2^T(t-\bar{\theta})Q_{4}\bar{x}_2(t-\bar{\theta})\nonumber\\&
-(1-\dot{\theta}(t))\bar{x}_2^T(t-\theta(t))\tilde{Q}_{4}\bar{x}_2(t-\theta(t))],\\
\label{lkf3}\mathbb{E}[\dot V_3]&=\mathbb{E}[\dot{\bar{x}}_1^T(\zeta_2^2R_1+d_2^2R_2+\tau_2^2R_3)\dot{\bar{x}}_1
\nonumber\\&-\zeta_2\int^{t}_{t-\zeta_2}\dot{\bar{x}}_1^T(s)R_1 \dot{\bar{x}}_1(s) ds
-d_2\int^{t}_{t-d_2}\dot{\bar{x}}_1^T(s)R_2 \dot{\bar{x}}_1(s) ds\nonumber\\&
-\tau_2\int^{t}_{t-\tau_2}\dot{\bar{x}}_1^T(s)R_3 \dot{\bar{x}}_1(s) ds],
\\
\label{lkf4}\mathbb{E}[\dot V_4]&=
\mathbb{E}[\bar{\theta} \bar{x}_2^TZ_1\bar{x}_2 - \int^{t}_{t-\bar{\theta}}\bar{x}_2^T(s)Z_1\bar{x}_2(s) ds\nonumber\\&
+ \bar{\theta}^2 \dot{\bar{x}}_2^T\mathcal{E}^TZ_2\mathcal{E} \dot{\bar{x}}_2
-\bar{\theta}\int^{t}_{t-\bar{\theta}}\dot{\bar{x}}_2^T(s)\mathcal{E}^TZ_2\mathcal{E} \dot{\bar{x}}_2(s) ds].
\end{align}}
{According to Lemma \ref{p1-writ1}, the term in \eqref{lkf3} can be written as
{\begin{align}
\label{p1-eq22}
-&\sum\limits_{i=1}^{3}r_i\int_{t-r_i(t)}^{t}
\dot{\bar{x}}_1^T(s)R_i\dot{\bar{x}}_1(s)ds \nonumber\\
&\leq -\frac{r_i}{r_i(t)}
\begin{bmatrix}
\Pi_{i1} \\ \Pi_{i2}
\end{bmatrix}^T
\begin{bmatrix}
R_i & 0 \\ \ast & 3R_i
\end{bmatrix}
\begin{bmatrix}
\Pi_{i1} \\ \Pi_{i2}
\end{bmatrix},\\
\label{p1-eq23}
-&\sum\limits_{i=1}^{3}r_i\int_{t-r_i}^{t-r_i(t)}
\dot{\bar{x}}_1^T(s)R_i\dot{\bar{x}}_1(s)ds\nonumber\\
&\leq -\frac{r_i}{r_i-r_i(t)}
\begin{bmatrix}
\Pi_{i3} \\ \Pi_{i4}
\end{bmatrix}^T
\begin{bmatrix}
R_i & 0 \\ \ast & 3R_i
\end{bmatrix}
\begin{bmatrix}
\Pi_{i3} \\ \Pi_{i4}
\end{bmatrix},
\end{align}}
where
{$\Pi_{i1}=\bar{x}_1-e_1\chi_i,$
$\Pi_{i2}=\bar{x}_2+e_1\chi_i-2e_3\chi_i,$
$\Pi_{i3}=e_1\chi_i-e_2\chi_i,$
$\Pi_{i4}=e_1\chi_i+e_2\chi_i-2e_4\chi_i$, $(i=1, 2, 3)$,
$\chi_1=\big[
\bar{x}_1^T(t-\zeta(t)) \ \ \bar{x}_1^T(t-\zeta_2) \ \ \frac{1}{\zeta(t)}\int_{t-\zeta(t)}^{t}\bar{x}_1^T(s)ds \ \
\frac{1}{\zeta_2-\zeta(t)}\int_{t-\zeta_2}^{t-\zeta(t)}\bar{x}_1^T(s)ds
\big]$,
$\chi_2=
\big[
\bar{x}_1^T(t-d(t)) \ \ \bar{x}_1^T(t-d_2) \ \
\frac{1}{d(t)}\int_{t-d(t)}^{t}\bar{x}_1^T(s)ds \ \
\frac{1}{d_2-d(t)}\int_{t-d_2}^{t-d(t)}\bar{x}_1^T(s)ds
\big]$,
$\chi_3=
\big[
\bar{x}_1^T(t-\tau(t)) \ \ \bar{x}_1^T(t-\tau_2) \ \ \frac{1}{\tau(t)}\int_{t-\tau(t)}^{t}\bar{x}_1^T(s)ds \ \
\frac{1}{\tau_2-\tau(t)}\int_{t-\tau_2}^{t-\tau(t)}\bar{x}_1^T(s)ds
\big]$,}
and $e_i\ (i=1, 2, 3, 4)$ are compatible row-block matrices with $i^{th}$ block of an identify matrix.\\
Now, by applying Lemma \ref{p1-writ2} in \eqref{p1-eq22} and \eqref{p1-eq23} can be written as
{\begin{align}\label{p1-eq24}
-&\sum\limits_{i=1}^{3}r_i\int_{t-r_i}^{t}\dot{\bar{x}}_1^T(s)R_i \dot{\bar{x}}_1(s) ds
\nonumber\\
&\leq
-\sum\limits_{i=1}^{3}
\begin{bmatrix}
\Pi_{i1} \\ \Pi_{i2} \\ \Pi_{i3} \\ \Pi_{i4}
\end{bmatrix}^T
\begin{bmatrix}
R_i & 0 & \mathbb{U}_{i1}^T & \mathbb{U}_{i3}^T \\
\ast & 3R_i & \mathbb{U}_{i2}^T  & \mathbb{U}_{i4}^T  \\
\ast & \ast &  R_i & 0 \\
\ast &  \ast &  \ast &  3R_i
\end{bmatrix}
\begin{bmatrix}
\Pi_{i1} \\ \Pi_{i2} \\ \Pi_{i3} \\ \Pi_{i4}
\end{bmatrix},
\end{align}}
where $\mathbb{U}_{1}=M$, $\mathbb{U}_{2}=N$, $\mathbb{U}_{3}=S$.\\
By using $0\leq \theta(t)\leq \bar{\theta}$ to the integral terms in $\eqref{lkf4}$, we can get
{\begin{align}
\label{p1-eq25}-\int_{t-\bar{\theta}}^{t}\bar{x}_2^{T}(s)Z_1\bar{x}_2(s)ds&=-\int_{t-\bar{\theta}}^{t-\theta(t)}\bar{x}_2^{T}(s)Z_1\bar{x}_2(s)ds
\nonumber\\&-\int_{t-\theta(t)}^{t}\bar{x}_2^{T}(s)Z_1\bar{x}_2(s)ds, \\
\label{p1-eq26}- \int_{t-\bar{\theta}}^{t}\dot{\bar{x}}_2^T(s)\mathcal{E}^TZ_2\mathcal{E} \dot{\bar{x}}_2(s) ds &= - \int_{t-\bar{\theta}}^{t-\theta(t)}\dot{\bar{x}}_2^T(s)\mathcal{E}^TZ_2\mathcal{E} \dot{\bar{x}}_2(s) ds
\nonumber\\&- \int_{t-\theta(t)}^{t}\dot{\bar{x}}_2^T(s)\mathcal{E}^TZ_2\mathcal{E} \dot{\bar{x}}_2(s) ds.
\end{align}}
Applying Lemma \ref{jnl} to the integral term in  \eqref{p1-eq25}, we obtain
{\begin{align}\label{p1-eq27}
-\int_{t-\bar{\theta}}^{t}\bar{x}_2^{T}(s)Z_1\bar{x}_2(s)ds & \leq  \frac{-1}{\bar{\theta}}\begin{bmatrix} \int_{t-\bar{\theta}}^{t-\theta(t)}\bar{x}_2(s)ds \\ \int_{t-\theta(t)}^{t}\bar{x}_2(s)ds \end{bmatrix}^T \begin{bmatrix} Z_1 & 0 \\ * & Z_1  \end{bmatrix} \nonumber\\ &\times
\begin{bmatrix} \int_{t-\bar{\theta}}^{t-\theta(t)}\bar{x}_2(s)ds \\ \int_{t-\theta(t)}^{t}\bar{x}_2(s)ds \end{bmatrix}.
\end{align}}
In addition, utilizing Lemma \ref{wrtl} to the integral terms of right hand side in \eqref{p1-eq26}, we can get the following inequalities
{\begin{align}
\label{p1-eq28}- &\int_{t-\bar{\theta}}^{t-\theta(t)}\dot{\bar{x}}_2^T(s)\mathcal{E}^TZ_2\mathcal{E} \dot{\bar{x}}_2(s) ds \nonumber\\& \leq  \frac{-1}{\bar{\theta}}\begin{bmatrix} \bar{x}_2(t-\theta(t)) \\ \bar{x}_2(t-\theta_2) \\ \frac{1}{\bar{\theta}}\int_{t-\bar{\theta}}^{t-\theta(t)}\bar{x}_2(s)ds\end{bmatrix}^T\mathcal{W} \begin{bmatrix} \bar{x}_2(t-\theta(t)) \\ \bar{x}_2(t-\bar{\theta}) \\ \frac{1}{\bar{\theta}}\int_{t-\bar{\theta}}^{t-\theta(t)}\bar{x}_2(s)ds\end{bmatrix},\\
\label{p1-eq29}- &\int_{t-\theta(t)}^{t}\dot{\bar{x}}_2^T(s)E^TZ_2E \dot{\bar{x}}_2(s) ds
\nonumber\\& \leq
\frac{-1}{\bar{\theta}}\begin{bmatrix} \bar{x}_2\\ \bar{x}_2(t-\theta(t)) \\ \frac{1}{\bar{\theta}}\int_{t-\theta(t)}^{t}\bar{x}_2(s)ds\end{bmatrix}^T\mathcal{W} \begin{bmatrix} \bar{x}_2 \\ \bar{x}_2(t-\theta(t)) \\ \frac{1}{\bar{\theta}}\int_{t-\theta(t)}^{t}\bar{x}_2(s)ds\end{bmatrix},
\end{align}}
where the matrix $\mathcal{W}$ is given as
\begin{eqnarray}
\mathcal{W} =
\begin{bmatrix}
\mathcal{E}^T Z_2 \mathcal{E}  & -\mathcal{E}^T Z_2\mathcal{E}  & 0 \\
\ast & \mathcal{E}^T Z_2 \mathcal{E} & 0\\
\ast & \ast & 0
\end{bmatrix} \nonumber \\
+\begin{bmatrix}
\frac{\pi^2}{4}\mathcal{E}^T Z_2 \mathcal{E} & \frac{\pi^2}{4}\mathcal{E}^T Z_2 \mathcal{E} & -\frac{\pi^2}{2}\mathcal{E}^T Z_2\mathcal{E} \\
\ast &\frac{\pi^2}{4}\mathcal{E}^T Z_2\mathcal{E} &-\frac{\pi^2}{2} \mathcal{E}^T Z_2 \mathcal{E} \\
\ast & \ast & \pi^2 \mathcal{E}^T Z_2 \mathcal{E}
\end{bmatrix}.
\end{eqnarray}
From \eqref{lkf1}, we can have
\begin{align*}
2\mathbb{E}\{\bar{x}_2^TP_2\mathcal{E}\dot{\bar{x}}_2\}=2\mathbb{E}\big\{\bar{x}_2^TP_2\big[\bar{\mathcal{A}}_0+\bar{\mathcal{A}}_1\big]\big\},
\end{align*}
where
${\mathcal{A}}_0=(A_2+B_2\mathcal{K}_2)\bar{x}_2+A_3\bar{x}_2(t-\theta(t))-B_2\Psi(u_2)+B_3w,\
\bar{\mathcal{A}}_1=(1-\bar{\beta})B_2\mathbb{G}\mathcal{K}_1\bar{x}_1(t-\zeta(t))+(1-\bar{\beta})(1-\bar{\alpha})B_2\mathbb{G}\mathcal{K}_1\bar{x}_1(t-d(t))
+(1-\bar{\beta})(1-\bar{\alpha})B_2\mathbb{G}\mathcal{K}_1e_k(t)+\bar{\beta}B_2\mathbb{G}\mathcal{K}_1f(\bar{x}_1(t-\tau(t))).
$\\}
Notice that
{\begin{align*}
\mathcal{E}\dot{x}_2(t)&=\mathcal{A}_0+(\alpha(t)-\bar{\alpha})\mathcal{A}_1+(\beta(t)-\bar{\beta})\mathcal{A}_2
\\&+(\alpha(t)-\bar{\alpha})(\beta(t)-\bar{\beta})\mathcal{A}_3,
\end{align*}}
where $
\mathcal{A}_0=\bar{\mathcal{A}}_0+\bar{\beta}B_2\mathbb{G}\mathcal{K}_1f(\bar{x}_1(t-\tau(t)))
+(1-\bar{\beta})\bar{\mathcal{A}}_2,\
\mathcal{A}_1=(1-\bar{\beta})\bar{\mathcal{A}}_3,\
\mathcal{A}_2=B_2\mathbb{G}\mathcal{K}_1f(\bar{x}_1(t-\tau(t)))-\bar{\mathcal{A}}_2,\
\mathcal{A}_3=-\bar{\mathcal{A}}_3,
\bar{\mathcal{A}}_2=\bar{\alpha}B_2\mathbb{G}\mathcal{K}_1\bar{x}_1(t-\zeta(t))+(1-\bar{\alpha})B_2\mathbb{G}\mathcal{K}_1\bar{x}_1(t-d(t))+(1-\bar{\alpha})B_2\mathbb{G}\mathcal{K}_1e_k(t),\
\bar{\mathcal{A}}_3=B_2\mathbb{G}\mathcal{K}_1\bar{x}_1(t-\zeta(t))-B_2\mathbb{G}\mathcal{K}_1\bar{x}_1(t-d(t))-B_2\mathbb{G}\mathcal{K}_1e_k(t).
$\\
In addition, it is clear that
{\begin{align}\label{p1-eq30}
\mathbb{E}\{\bar{\theta}^2 \dot{\bar{x}}_2^T&\mathcal{E}^TZ_2\mathcal{E} \dot{\bar{x}}_2\}=
\bar{\theta}^2\mathcal{A}_0^TZ_2\mathcal{A}_0
+\bar{\theta}^2\sigma^2\mathcal{A}_1^TZ_2\mathcal{A}_1\nonumber\\&
+\bar{\theta}^2\delta^2\mathcal{A}_2^TZ_2\mathcal{A}_2
+\bar{\theta}^2\sigma^2\delta^2\mathcal{A}_3^TZ_2\mathcal{A}_3.
\end{align}}
From the inequality \eqref{p1-eq9} and  Assumption 1 in \cite{base}, we can get the following inequalities
{\begin{align}
\label{p1-eq31}\mu \bar{x}_1^T(t-d(t))W\bar{x}_1(t-d(t))-e_k^TWe_k&\geq 0,\
\end{align}
and
\begin{align}
\label{p1-eq32}
&\bar{\beta}\bar{x}_1^T(t-\tau(t))F^TF\bar{x}_1(t-\tau(t))\nonumber\\&
-\bar{\beta}f^T(\bar{x}_1(t-\tau(t)))f(\bar{x}_1(t-\tau(t)))\geq 0.
\end{align}}
Note that
{\begin{align*}
u_2&=\bar{\mathcal{B}}_0+(\alpha(t)-\bar{\alpha})\bar{\mathcal{B}}_1
+(\beta(t)-\bar{\beta})\bar{\mathcal{B}}_2\\&
+(\alpha(t)-\bar{\alpha})(\beta(t)-\bar{\beta})\bar{\mathcal{B}}_3,
\end{align*}}
where
$
\bar{\mathcal{B}}_0=(1-\bar{\beta})\bar{\alpha}\mathbb{G}\mathcal{K}_1\bar{x}_1(t-\zeta(t))
+(1-\bar{\beta})(1-\bar{\alpha})\mathbb{G}\mathcal{K}_1\bar{x}_1(t-d(t))
+\mathcal{K}_2\bar{x}_2(t)\bar{\beta}\mathbb{G}\mathcal{K}_1f(\bar{x}_1(t-\tau(t)))+(1-\bar{\beta})(1-\bar{\alpha})\mathbb{G}\mathcal{K}_1e_k,\
\bar{\mathcal{B}}_1=(1-\bar{\beta})\mathbb{G}\mathcal{K}_1\bar{x}_1(t-\zeta(t))-(1-\bar{\beta})\mathbb{G}\mathcal{K}_1\bar{x}_1(t-d(t))-(1-\bar{\beta})\mathbb{G}\mathcal{K}_1e_k,\
\bar{\mathcal{B}}_2=-\bar{\alpha}\mathbb{G}\mathcal{K}_1\bar{x}_1(t-\zeta(t))-(1-\bar{\alpha})\mathbb{G}\mathcal{K}_1\bar{x}_1(t-d(t))+\mathbb{G}\mathcal{K}_1f(\bar{x}_1(t-\tau(t)))-(1-\bar{\alpha})\mathbb{G}\mathcal{K}_1e_k,\
\bar{\mathcal{B}}_3=-\mathbb{G}\mathcal{K}_1\bar{x}_1(t-\zeta(t))+\mathbb{G}\mathcal{K}_1\bar{x}_1(t-d(t)+\mathbb{G}\mathcal{K}_1e_k.
$\\
From \eqref{p1-eq5}, it is easy to obtain that
{\begin{align}\label{p1-eq33}
\mathbb{E}\{\epsilon u_2^Tu_2\}&=
\epsilon \bar{\mathcal{B}}_0^T\bar{\mathcal{B}}_0+\sigma^2\epsilon\bar{\mathcal{B}}_1^T\bar{\mathcal{B}}_1\nonumber\\&
+\delta^2\epsilon\bar{\mathcal{B}}_2^T\bar{\mathcal{B}}_2+\sigma^2\delta^2\epsilon\bar{\mathcal{B}}_3^T\bar{\mathcal{B}}_3.
\end{align}}
Now, by combining \eqref{p1-eq5}, \eqref{lkf1}-\eqref{p1-eq33} and Schur complement with $w=0$, we get
\begin{align}\label{p1-eq34}
\dot{V}(t)\leq \eta^T(t)\hat{\Omega}
\eta(t)
\end{align}
where
$\eta(t)=\big[ \bar{x}_1^T\ \ \chi_1^T  \ \ \chi_2^T \ \ \chi_3^T \ \ \bar{x}_2^T \ \ \bar{x}_2^T(t-\theta(t))\ \ \bar{x}_2^T(t-\bar{\theta})  \ \
\int_{t-\theta(t)}^{t}\bar{x}_2^T(s)ds
\ \
\int_{t-\bar{\theta}}^{t-\theta(t)}\bar{x}_2^T(s)ds  \ \
 e_k^T
\ \ f^T(\bar{x}_1(t-\tau(t)))\ \ \Psi^T(u_2) \big]^T.$
$\hat{\Omega}=[{\Omega}]_{21\times 21}
+\zeta_2^2{\Omega}_1^TR_1{\Omega}_1
+d_2^2{\Omega}_1^TR_2{\Omega}_1
+\tau_2^2{\Omega}_1^TR_3{\Omega}_1
+\bar{\theta}^2{\Omega}_2^TZ_2{\Omega}_2
+\bar{\theta}^2{\Omega}_3^TZ_2{\Omega}_3
+\bar{\theta}^2{\Omega}_4^TZ_2{\Omega}_4
+\bar{\theta}^2{\Omega}_5^TZ_2{\Omega}_5
+\epsilon{\Omega}_6^T{\Omega}_6
+\epsilon\sigma^2{\Omega}_7^T{\Omega}_7
+\epsilon\delta^2{\Omega}_8^T{\Omega}_8
+\epsilon\sigma^2\delta^2{\Omega}_9^T{\Omega}_9
+{\Omega}_{10}^T{\Omega}_{10}.$
Then, by using Lemma \ref{p1-schur} to the inequality \eqref{p1-eq34}, it is easy to get
that
\begin{align}\label{p1-eq35}
\tilde{\Omega}=
\begin{bmatrix}
[{\Omega}]_{21\times 21} & \tilde{\tilde{\Omega}}_1\\
\ast & \tilde{\tilde{\Omega}}_2
\end{bmatrix}<0,
\end{align}
where $\tilde{\Omega}$ elements are
{$\Omega_{1,1}=2P_1A_1+Q_1+Q_2+Q_3-4R_1-4R_2-4R_3$,
$\Omega_{1,2}=-2R_1-(M_1+M_2+M_3+M_4)^T$,
$\Omega_{1,3}=(M_1+M_2-M_3-M_4)^T$,
$\Omega_{1,4}=6R_1$,
$\Omega_{1,5}=2(M_3+M_4)^T$,
$\Omega_{1,6}=-2R_2-(N_1+N_2+N_3+N_4)^T$,
$\Omega_{1,7}=(N_1+N_2-N_3-N_4)^T$,
$\Omega_{1,8}=6R_2$,
$\Omega_{1,9}=2(N_3+N_4)^T$,
$\Omega_{1,10}=-2R_3-(S_1+S_2+S_3+S_4)^T$,
$\Omega_{1,11}=(S_1+S_2-S_3-S_4)^T$,
$\Omega_{1,12}=6R_3$,
$\Omega_{1,13}=2(S_3+S_4)^T$,
$\Omega_{1,14}=P_1B_1C_2$,
$\Omega_{2,2}=-8R_1+2(M_1-M_2+M_3-M_4)^T$,
$\Omega_{2,3}=-2R_1+(-M_1+M_2+M_3-M_4)^T$,
$\Omega_{2,4}=6R_1+2(M_2+M_4)^T$,
$\Omega_{2,5}=6R_1-2M_3^T+2M_4^T$,
$\Omega_{2,14}=(1-\bar{\beta})\bar{\alpha}\mathcal{K}_1^T\mathbb{G}^TB_2^TP_2$,
$\Omega_{3,3}=-4R_1-Q_1$,
$\Omega_{3,4}=-2M_2^T+2M_4^T$,
$\Omega_{3,4}=6R_1$,
$\Omega_{4,4}=-12R_1$,
$\Omega_{4,5}=-4M_4^T$,
$\Omega_{5,5}=-12R_1$,
$\Omega_{6,6}=-8R_2+2(N_1-N_2+N_3-N_4)^T+\mu W$,
$\Omega_{6,7}=-2R_2+(-N_1+N_2+N_3-N_4)^T$,
$\Omega_{6,8}=6R_2+2N_2^T+2N_4^T$,
$\Omega_{6,9}=6R_2-2N_3^T+2N_4^T$,
$\Omega_{6,14}=(1-\bar{\beta})(1-\bar{\alpha})\mathcal{K}_1^T\mathbb{G}^TB_2^TP_2$,
$\Omega_{7,7}=-4R_2-Q_2$,
$\Omega_{7,8}=-2N_2^T+2N_4^T$,
$\Omega_{7,9}=6R_2$,
$\Omega_{8,8}=-12R_2$,
$\Omega_{8,9}=-4N_4^T$,
$\Omega_{9,9}=-12R_2$,
$\Omega_{10,10}=-8R_3+2(S_1-S_2+S_3-S_4)^T$,
$\Omega_{10,11}=-2R_3+(-S_1+S_2+S_3-S_4)^T$,
$\Omega_{10,12}=6R_3+2S_2^T+2S_4^T$,
$\Omega_{10,13}=6R_3-2S_3^T+2S_4^T$,
$\Omega_{11,11}=-4R_3-Q_3$,
$\Omega_{11,12}=-2S_2^T+2S_4^T$,
$\Omega_{11,13}=6R_3$,
$\Omega_{12,12}=-12R_3$,
$\Omega_{12,13}=-4S_4^T$,
$\Omega_{13,13}=-12R_3$,
$\Omega_{14,14}=Q_4+\tilde{Q}_4+2P_2A_2+2P_2B_2\mathcal{K}_2-\mathcal{E}^TZ_2\mathcal{E}-\frac{\pi^2}{4}\mathcal{E}^TZ_2\mathcal{E}$,
$\Omega_{14,15}=P_2A_3+\mathcal{E}^TZ_2\mathcal{E}-\frac{\pi^2}{4}\mathcal{E}^TZ_2\mathcal{E}$,
$\Omega_{14,17}=\frac{\pi^2}{4}\mathcal{E}^TZ_2\mathcal{E}$,
$\Omega_{14,19}=(1-\bar{\beta})(1-\bar{\alpha})P_2B_2\mathbb{G}\mathcal{K}_1$,
$\Omega_{14,20}=\bar{\beta}P_2B_2\mathbb{G}\mathcal{K}_1$,
$\Omega_{14,21}=-P_2B_2$,
$\Omega_{15,15}=-2\mathcal{E}^TZ_2\mathcal{E}-2\frac{\pi^2}{4}\mathcal{E}^TZ_2\mathcal{E}-(1-\lambda)\hat{Q}_4$,
$\Omega_{15,16}=\mathcal{E}^TZ_2\mathcal{E}-\frac{\pi^2}{4}\mathcal{E}^TZ_2\mathcal{E}$,
$\Omega_{15,17}=\frac{\pi^2}{4}\mathcal{E}^TZ_2\mathcal{E}$,
$\Omega_{15,18}=\frac{\pi^2}{4}\mathcal{E}^TZ_2\mathcal{E}$,
$\Omega_{16,16}=-Q_4-\mathcal{E}^TZ_2\mathcal{E}-\frac{\pi^2}{4}\mathcal{E}^TZ_2\mathcal{E}$,
$\Omega_{16,18}=\frac{\pi^2}{4}E^TZ_2E$,
$\Omega_{17,17}=-\pi^2E^TZ_2E-\frac{1}{\bar{\theta}}Z_1$,
$\Omega_{18,18}=-\pi^2E^TZ_2E-\frac{1}{\bar{\theta}}Z_1$,
$\Omega_{19,19}=-WI$,
$\Omega_{20,20}=-\bar{\beta}I$,
$\Omega_{21,21}=-I$,
$\tilde{\tilde{\Omega}}_1=
\big[\zeta_2{\Omega}_1^T \ \
d_2{\Omega}_1^T \ \
\tau_2{\Omega}_1^T \ \
\theta_2{\Omega}_2^T \ \
\theta_2{\Omega}_3^T \ \
\theta_2{\Omega}_4^T \ \
\theta_2{\Omega}_5^T \ \
\sqrt{\epsilon}{\Omega}_6^T \\
\sqrt{\epsilon}\sigma{\Omega}_7^T \ \
\sqrt{\epsilon}\delta{\Omega}_8^T \ \
\sqrt{\epsilon}\sigma\delta{\Omega}_9^T\ \
{\Omega}_{10}^T\big]$,\
$\tilde{\tilde{\Omega}}_2=\mbox{diag}\big\{-\kappa_1, -\kappa_2, -\kappa_3, -\kappa_4, -\kappa_4, -\kappa_4, -\kappa_4, -I, -I, -I, -I, -I\big\}$,
$\Omega_1=
\big[
P_1A_1 \ \ 0_{12n} \ \ P_1B_1C_2 \ \ 0_{7n}
\big]$, \
$\Omega_2=
\big[
0 \ \bar{\alpha}\bar{\beta}_1\Upsilon_1 \ 0_{3n} \ \bar{\alpha}_1\bar{\beta}_1\Upsilon_1 \ 0_{7n} \  P_2A_2+\Upsilon_2 \ P_2A_3 \ 0_{3n} \  \bar{\alpha}_1\bar{\beta}_1\Upsilon_1 \ \bar{\beta}\Upsilon_1 \ -P_2B_2
\big]$,
$\Omega_3=
\big[
0 \ \ \sigma\bar{\beta}_1\Upsilon_1 \ \ 0_{3n} \ \ \sigma\bar{\beta}_1\Upsilon_1 \ \ 0_{12n} \ \ \sigma\bar{\beta}_1\Upsilon_1 \ \  0_{2n}
\big]$,\
$\Omega_4=
\big[
0 \ \ -\delta\bar{\alpha}_1\Upsilon_1 \ \ 0_{3n} \ \ -\delta\bar{\alpha}_1\Upsilon_1  \ \ 0_{12n} \ \ -\delta\bar{\alpha}_1\Upsilon_1  \ \ \delta\bar{\alpha}_1\Upsilon_1 \ \ 0
\big]$, \ \
$\Omega_5=
\big[
0 \ \ \sigma\delta \Upsilon_1 \ \ 0_{3n} \ \ -\sigma\delta \Upsilon_1 \ \ 0_{12n} \ \ -\sigma\delta \Upsilon_1 \ \  0_{2n}
\big]$,\
${\Omega}_6=
\big[
0 \ \bar{\beta}_1\bar{\alpha}\mathbb{G}\mathcal{K}_1 \ 0_{3n} \ \bar{\beta}_1\bar{\alpha}_1\mathbb{G}\mathcal{K}_1 \ 0_{7n} \ K_2 \ 0_{4n} \ \bar{\beta}_1\bar{\alpha}_1\mathbb{G}\mathcal{K}_1 \ \bar{\beta}\mathbb{G}\mathcal{K}_1 \ 0
\big]$,
${\Omega}_7=
\big[
0 \ \ \bar{\beta}_1\mathbb{G}\mathcal{K}_1 \ \ 0_{3n} \ \ -\bar{\beta}_1\mathbb{G}\mathcal{K}_1 \ \ 0_{12n} \ \ -\bar{\beta}_1\mathbb{G}\mathcal{K}_1  \ \ 0_{2n}
\big]$,\
${\Omega}_8=
\big[
0 \ \ -\bar{\alpha}\mathbb{G}\mathcal{K}_1 \ \ 0_{3n} \ \ -\bar{\alpha}_1\mathbb{G}\mathcal{K}_1 \ \ 0_{12n} \ \ -\bar{\alpha}_1\mathbb{G}\mathcal{K}_1 \ \ \mathbb{G}\mathcal{K}_1\ \ 0
\big]$,\
${\Omega}_9=
\big[
0 \ \ -\mathbb{G}\mathcal{K}_1 \ \ 0_{3n} \ \ \mathbb{G}\mathcal{K}_1 \ \ 0_{12n} \ \ \mathbb{G}\mathcal{K}_1 \ \  0_{2n}
\big]$,\
${\Omega}_{10}=
\big[
0_{9n}\ \ \sqrt{\bar{\beta}}F \ \ 0_{11n}
\big]$,
$\Upsilon_1=P_2B_2\mathbb{G}\mathcal{K}_1$,
$\Upsilon_2=P_2B_2\mathcal{K}_2$,\
$\kappa_1=P_1R_1^{-1}P_1$,
$\kappa_2=P_1R_2^{-1}P_1$,
$\kappa_3=P_1R_3^{-1}P_1$,
$\kappa_4=P_2Z_2^{-1}P_2$.
}
Next, we discuss the dissipativity of the augmented system \eqref{p1-eq14} with non-zero disturbances.
For the given disturbance attenuation level $\gamma>0$. For this, we introduce the following performance index:
\begin{align}\label{p1-eq36}
\mathbb{J}=\int_{0}^{t}[-\bar{y}_1^T(s)\mathcal{Q}\bar{y}_1(s)&-2y_1^T(s)\mathcal{R}w(s)-w^T(s)\mathcal{S}w(s)
\nonumber\\
&+\gamma w^T(s)w(s)]ds.
\end{align}
Using output vector $y_1$ defined in \eqref{p1-eq14} and from \eqref{p1-eq35}, we get
\begin{align}\label{p1-eq37}
\dot{V}(t)+\mathbb{J}\leq \xi^T(t)\begin{bmatrix}
[{\Omega}]_{22\times 22} & \bar{\bar{\Omega}}_1\\
\ast & \bar{\bar{\Omega}}_2
\end{bmatrix}\xi(t)
\end{align}
where the elements of right hand side of \eqref{p1-eq37} are given as
{
$\xi(t)=
\begin{bmatrix}
\eta^T(t) & w^T
\end{bmatrix}$,
$\Omega_{1,22}=P_1B_1D_2-C_1^T\mathcal{S}$,
$\Omega_{14,22}=-P_2B_3$,
$\Omega_{22,22}=-\mathcal{R}-D_2^T\mathcal{S}+\gamma I$,\
$\bar{\bar{\Omega}}_1=
\big[
\zeta_2\bar{{\Omega}}_1^T \ \
d_2\bar{{\Omega}}_1^T \ \
\tau_2\bar{{\Omega}}_1^T \ \
\theta_2\bar{{\Omega}}_2^T \ \
\theta_2\bar{{\Omega}}_3^T \ \
\theta_2\bar{{\Omega}}_4^T \ \
\theta_2\bar{{\Omega}}_5^T \ \
\sqrt{\epsilon}\bar{{\Omega}}_6^T \ \
\sqrt{\epsilon}\sigma\bar{{\Omega}}_7^T \\ \
\sqrt{\epsilon}\delta\bar{{\Omega}}_8^T \ \
\sqrt{\epsilon}\sigma\delta\bar{{\Omega}}_9^T\ \ \
\bar{{\Omega}}_{10}^T \ \
\bar{{\Omega}}_{11}^T\big]$,\ \
$\bar{\bar{\Omega}}_2=\mbox{diag}\big\{-\bar{\kappa}_1, -\bar{\kappa}_2, -\bar{\kappa}_3, -\bar{\kappa}_4, -\bar{\kappa}_4, -\bar{\kappa}_4, -\bar{\kappa}_4, -I, -I, -I, -I, -I, -I\big\}$,\
$\bar{\Omega}_1=
\big[
{\Omega}_1 \ \ P_1B_1D_2
\big]$,\
$\bar{\Omega}_2=
\big[
{\Omega}_2 \ \  P_2B_3
\big]$,
$\bar{\Omega}_3=
\big[
{\Omega}_3 \ \  0
\big]$,\
$\bar{\Omega}_4=
\big[
{\Omega}_4 \ \  0
\big]$,
$\bar{\Omega}_5=
\big[
{\Omega}_5 \ \  0
\big]$,\
$\bar{\Omega}_6=
\big[
{\Omega}_6 \ \ 0
\big]$,\
$\bar{\Omega}_7=
\big[
{\Omega}_7 \ \ 0
\big]$,\
$\bar{\Omega}_8=
\big[
{\Omega}_8 \ \ 0
\big]$,\
$\bar{\Omega}_9=
\big[
{\Omega}_9 \ \ 0
\big]$,
$\bar{\Omega}_{10}=
\big[
{\Omega}_{10} \ \ 0
\big]$,
$\bar{\Omega}_{11}=
\big[
\sqrt{\mathcal{Q}}C_1X_1^T \ \ 0_{20n} \ \ \sqrt{\mathcal{Q}}D_1
\big]$,
$\bar{\kappa}_1=-2\epsilon_1P_1+\epsilon_1^2R_1$,
$\bar{\kappa}_2=-2\epsilon_2P_1+\epsilon_2^2R_2$,
$\bar{\kappa}_3=-2\epsilon_3P_1+\epsilon_3^2R_3$,
$\bar{\kappa}_4=-2\epsilon_1P_2+\epsilon_4^2Z_2$.}
In order to complete the proof, pre- and post- multiplying \eqref{p1-eq24} and \eqref{p1-eq37} by
{$\mbox{diag}\{X_1, X_1, X_1, X_1\}$} and {$\mbox{diag}\{\underbrace{X_1, \cdots, X_1}_{13}, \underbrace{X_2, \cdots, X_2}_{5}, X_1, X_1, I, I, X_1, X_1, X_1, X_2, X_2, X_2, X_2, \\ \underbrace{I, I, \cdots, I}_{6}\}$}, respectively.
Letting
{$P_1=X_1^{-1}$, $P_2=X_2^{-1}$
$\hat{Q}_i=X_1Q_iX_1,\ (i=1, 2, 3)$, $\hat{Q}_4=X_2{Q}_4X_2$, $\hat{\tilde{Q}}_4=X_2\tilde{Q}_4X_2$, $\hat{R}_j=X_1R_jX_1,\ (j=1, 2, 3)$,
$\hat{Z}_1=X_2Z_1X_2$, $\hat{Z}_2=X_2Z_2X_2$, $\hat{\mathbb{U}}_{i1}=X_1\mathbb{U}_{i1}X_1$, $\hat{\mathbb{U}}_{i2}=X_1\mathbb{U}_{i2}X_1$, $\hat{\mathbb{U}}_{i3}=X_1\mathbb{U}_{i3}X_1$, $\hat{\mathbb{U}}_{i4}=X_1\mathbb{U}_{i4}X_1$ $(i=1, 2, 3)$, $\mathbb{U}_{1}=M$, $\mathbb{U}_{2}=N$, $\mathbb{U}_{3}=S$.}
We can get LMIs \eqref{p1-eq24} and \eqref{p1-eq37} are equivalent to the LMIs \eqref{p1-thm1lmi1} and \eqref{p1-thm1lmi2},
respectively. This implies that the closed-loop singular NCCSs \eqref{p1-eq14} is mean-square asymptotically stable.\

Next, we will prove the regularity and the impulse-free condition for the system \eqref{p1-eq14}. Here, we assume that the matrix $\mathcal{E}$ and the state vector $\bar{x}_2$ in the following forms:
\begin{align*}
\mathcal{E}=
\begin{bmatrix}
I_r & 0_{r\times (n-r)}\\
0_{(n-r)\times r} & 0_{(n-r)}
\end{bmatrix}\quad \mbox{and} \quad
\bar{x}_2=
\begin{bmatrix}
\bar{x}_{21} \\ \bar{x}_{22}
\end{bmatrix}
\end{align*}
where $\bar{x}_{21}\in\mathbb{R}^r$ and $\bar{x}_{22}\in \mathbb{R}^{(n-r)}$.\\
It follows from \eqref{p1-eq35} that
\begin{align}\label{p1-proofimp}
\Omega_{14,14}=Q_4+\tilde{Q}_4+2P_2A_2.
\end{align}
Next, we define
{\begin{align*}
P_2&=
\begin{bmatrix}
P_{21} & P_{22}\\P_{23} & P_{24}
\end{bmatrix},\quad
A_{2}=
\begin{bmatrix}
A_{21} & A_{22}\\ A_{23} & A_{24}
\end{bmatrix}, \\
Q_4&=
\begin{bmatrix}
Q_{41} & Q_{42}\\Q_{43} & Q_{44}
\end{bmatrix},\quad
\tilde{Q}_4=
\begin{bmatrix}
\tilde{Q}_{41} & \tilde{Q}_{42}\\ \tilde{Q}_{43} & \tilde{Q}_{44}
\end{bmatrix}.
\end{align*}}
By substituting $P_2$, $A_2$, $Q_4$ and $\tilde{Q}_4$ into \eqref{p1-proofimp}, it yields that
\begin{eqnarray*}
(P_{24}A_{24}+A_{24}^TP_{24}^T)+Q_{44}+ \tilde{Q}_{44}<0.
\end{eqnarray*}
Then, it is easy to find
$
\mbox{det}(s\hat{\mathcal{E}}-\hat{A}_2)=\mbox{det}(s\mathcal{E}-A_2)
$
which tends to $\mbox{det}(s\mathcal{E}-A_2)$ is not identically zero and $\mbox{det}(s\mathcal{E}-A_2)=r=\mbox{rank}(\mathcal{E}).$ Thus, the system \eqref{p1-eq14} is  regular and impulse free. Hence by Definition \ref{p1-def1}, the systems \eqref{p1-eq14} is mean-square asymptotically admissible. \
\end{proof}

\subsection{Theorem~\ref{p1-thm2}}\label{sec:app:3}
{$\tilde{B}=
\big[
\tilde{\epsilon}_1\tilde{B}_1^T \ \ \tilde{Y}_1^T \ \
\tilde{\epsilon}_2\tilde{B}_2^T \ \ \tilde{Y}_2^T \ \
\tilde{\epsilon}_3\tilde{B}_3^T \ \ \tilde{Y}_3^T \ \
\tilde{\epsilon}_4\tilde{B}_4^T \ \ \tilde{Y}_4^T \ \
\tilde{\epsilon}_5\tilde{B}_5^T \ \ \tilde{Y}_5^T \\
\tilde{\epsilon}_6\tilde{B}_6^T \ \ \tilde{Y}_6^T \ \
\tilde{\epsilon}_7\tilde{B}_7^T \ \ \tilde{Y}_7^T \ \
\tilde{\epsilon}_8\tilde{B}_8^T \ \ \tilde{Y}_8^T
\big]$,
$\tilde{\epsilon}=\mbox{diag}
\big\{
-\tilde{\epsilon}_1 \ \ -\tilde{\epsilon}_1 \ \
-\tilde{\epsilon}_2 \ \ -\tilde{\epsilon}_2 \ \
-\tilde{\epsilon}_3 \ \ -\tilde{\epsilon}_3 \ \
-\tilde{\epsilon}_4 \ \ -\tilde{\epsilon}_4 \ \
-\tilde{\epsilon}_5 \ \ -\tilde{\epsilon}_5 \ \ \
-\tilde{\epsilon}_6 \ \ -\tilde{\epsilon}_6 \ \
-\tilde{\epsilon}_7 \ \ -\tilde{\epsilon}_7 \ \
-\tilde{\epsilon}_8 \ \ -\tilde{\epsilon}_8
\big\}$,
$\tilde{B}_1=
\big[
0 \ \ B_2 \ \ 0_{33n}
\big]$,
$\tilde{B}_2=
\big[
0_{5n} \ \ B_2 \ \ 0_{29n}
\big]$,
$\tilde{B}_3=
\big[
0_{18n} \ \ B_2 \ \ 0_{16}
\big]$,
$\tilde{B}_4=
\big[
0_{19n} \ \ B_2 \ \ 0_{15}
\big]$,
{$\tilde{B}_5=
\big[
0_{29n} \ \ \bar{\beta}_1I \ \ \bar{\beta}_1I
\ \ -\bar{\alpha}_1I \ \ -I \ \ 0_{2n}
\big]$,
$\tilde{B}_6=
\big[
0_{29n} \ \ \bar{\beta}_1\bar{\alpha}_1I\ \ -\bar{\beta}_1I
\ \ -\bar{\alpha}_1I \ \ I \ \ 0_{2n}
\big]$,\
$\tilde{B}_7=
\big[
0_{29n} \ \ \bar{\beta}_1\bar{\alpha}_1I \ \ -\bar{\beta}_1I
\ \ -\bar{\alpha}_1I \ \ I \ \ 0_{2n}
\big]$,\
$\tilde{B}_8=
\big[
0_{29n} \ \ \bar{\beta}_1I \ \ 0
\ \ I  \ \ 0_{3n}
\big]$,}\
$\tilde{Y}_1=
\big[
0_{25n} \ \ \bar{\alpha}\bar{\beta}_1\mathbb{G}_1Y_1 \ \ \sigma\bar{\beta}_1\mathbb{G}_1Y_1
\ \ -\delta\bar{\alpha}_1\mathbb{G}_1Y_1 \ \ \sigma\delta \mathbb{G}_1Y_1 \ \ 0_{5n}
\big]$,\
$\tilde{Y}_2=
\big[
0_{25n} \ \ \bar{\alpha}_1\bar{\beta}_1\mathbb{G}_1Y_1 \ \ \sigma\bar{\beta}_1\mathbb{G}_1Y_1
\ \ -\delta\bar{\alpha}_1\mathbb{G}_1Y_1 \ \ -\sigma\delta \mathbb{G}_1Y_1 \ \ 0_{5n}
\big]$,\
$\tilde{Y}_3=
\big[
0_{25n} \ \ -\bar{\alpha}_1\bar{\beta}_1\mathbb{G}_1Y_1 \ \ \sigma\bar{\beta}_1\mathbb{G}_1Y_1
\ \ -\delta\bar{\alpha}_1\mathbb{G}_1Y_1 \ \ -\sigma\delta \mathbb{G}_1Y_1 \ \ 0_{5n}
\big]$,\
$\tilde{Y}_4=
\big[
0_{25n} \ \ \bar{\beta}\mathbb{G}_1Y_1 \ \ 0 \ \  \delta\bar{\alpha}_1\mathbb{G}_1Y_1 \ \  0_{6n}
\big]$,
{$\tilde{Y}_5=
\big[
0 \ \ \mathbb{G}_1Y_1 \ \ 0_{33n}
\big]$,
$\tilde{Y}_6=
\big[
0_{5n} \ \ \mathbb{G}_1Y_1 \ \ 0_{29n}
\big]$,\
$\tilde{Y}_7=
\big[
0_{18n} \ \ \mathbb{G}_1Y_1 \ \ 0_{16}
\big]$,
$\tilde{Y}_8=
\big[
0_{19n} \ \ \mathbb{G}_1Y_1 \ \ 0_{15}
\big]$.}

\addtolength{\textheight}{-12cm}   




\end{document}